\documentclass[11pt]{amsart}
\usepackage{amsmath,amsxtra,amssymb,amsthm,amsfonts}
\usepackage{color}

\numberwithin{equation}{section}
\newtheorem{lemma}{Lemma}[section]
\newtheorem{prop}[lemma]{Proposition}
\newtheorem{theorem}[lemma]{Theorem}
\newtheorem{cor}[lemma]{Corollary}

\newtheorem{prob}[lemma]{Problem}
\newtheorem{rem}[lemma]{Remark}

\newcommand{\re}{\begin{rem}\rm}
  \newcommand{\mar}{\end{rem}}
\newtheorem{exam}[lemma]{Example}

\newtheorem{defi}[lemma]{Definition}

\newcommand{\p}{\hspace{.05cm}}
\newcommand{\pl}{\hspace{.1cm}}

\newcommand{\kl}{\pl \le \pl}
\newcommand{\gl}{\pl \ge \pl}

\newcommand{\lel}{\pl = \pl}

\newcommand{\kla}{\left ( }
\newcommand{\mer}{\right ) }
\def\1{\mathbf{1}}

\newcommand{\C}{\mathbb{C}}
\renewcommand{\for}{\begin{eqnarray*}}

\newcommand{\mel}{\end{eqnarray*}}

\newcommand{\ez}{{\mathbb E}}

\newcommand{\nz}{{\mathbb N}}

\newcommand{\rz}{{\mathbb R}}

\newcommand{\ten}{\otimes}

\DeclareMathOperator{\dom}{dom}

\newcommand{\qd}{\end{proof}\vspace{0.5ex}}

\newcommand{\Om}{\Omega}
\newcommand{\om}{\omega}

\newcommand{\si}{\sigma}

\newcommand{\la}{\lambda}
\newcommand{\eps}{\varepsilon}

\newcommand{\E}{{\mathcal E}}
\newcommand{\A}{{\mathcal A}}

\newcommand{\M}{{\mathcal M}}

\renewcommand{\P}{\mathcal P}

\newcommand{\N}{{\mathcal N}}

\newcommand{\T}{\mathcal T}

\newcommand{\U}{{\mathcal U}}

\newcommand{\xspace}{\hbox{\kern-2.5pt}}
\newcommand{\xyspace}{\hbox{\kern-1.1pt}}

\allowdisplaybreaks

\newcommand{\ttt}{{\bf t}}

\begin{document}
\title{BMO spaces associated with semigroups of operators}

\author[M. Junge]{M. Junge}
\address{Department of Mathematics\\
University of Illinois, Urbana, IL 61801, USA} \email[Marius
Junge]{junge@math.uiuc.edu}

\author[T. Mei]{T. Mei}
\address{Department of Mathematics\\
Wayne state University, Detroit, MI, 48202, USA} \email[Tao
Mei]{mei@wayne.edu}

\begin{abstract} We study BMO spaces associated with semigroup
of operators on noncommutative function spaces (i.e. von Neumann
algebras) and apply the results to boundedness of Fourier
multipliers on non-abelian discrete groups. We prove an
interpolation theorem for BMO spaces and prove the boundedness of a
class of Fourier multipliers on noncommutative $L_p$ spaces for all
$1<p<\infty$, with optimal constants in $p$.
\end{abstract}


\thanks{ The first author is partially supported by the
NSF DMS -090145705. The second author is partially supported by
NSF DMS-0901009.}

\maketitle {\small {\bf Mathematics subject classification} (2000):
46L51 (42B25 46L10 47D06) \medskip

 }

\section*{Introduction} The theory of semigroups provides a good framework of
studying classical questions from harmonic analysis in a more
abstract setting. Our research is particularly motivated by E.
Steins' results on Fourier multipliers on  $L_p$ spaces and
Littlewood-Paley theory for the Laplace-Beltrami operators on
compact groups. Our aim is to study BMO spaces which are
intrinsically defined by a (some kind of heat-) semigroup and prove
fundamental interpolation results.  In particular, we want to give a
positive answer to the following

  \begin{prob}\label{mainpb}  Let $(T_t)$ be a standard semigroup of selfadjoint positive
operators on an (abstract) functions space $L_\infty(\Omega)$. Let
$A$ be its infinitesimal generator. Is there a $BMO$ space such that
   \begin{enumerate}
   \item[(a)] $BMO$ serves as an endpoint of interpolation, i.e.
$[BMO,L_1(\Omega)]_{\frac1p}=L_p(\Omega)$;
   \item[(b)] The imaginary powers $A^{is}$, $s\in \rz$ extend to bounded
operators from $L_{\infty}(\Omega)$ to $BMO$.
   \item[(c)] The estimates in (b) are universal. In particular, the constants
involved are dimension free for all the classical heat semigroups on
$\Omega=\rz^n$.
  \end{enumerate}
  \end{prob}

We should expect much more singular integral operators for abstract
semigroups instead of the imaginary powers mentioned in (b).
However, it seems that even in the commutative theory such a $BMO$
space  has not yet been identified. An advantage of such a theory is
that it provides a natural framework for good, or even optimal
dimension free estimates, for Fourier multipliers. Our results apply
not only in the commutative, but also in the noncommutative setting
(i.e. replacing $L_\infty(\Omega)$ by a von Neumann algebra).

Indeed, BMO spaces, once they can be appropriately
defined, provide a very efficient tool in proving results on Fourier-multipliers. BMO spaces associated with semigroups on commutative functions spaces  have been
studied in \cite{Vo}, \cite{SV} and very recently in \cite{DuLi},
\cite{DuLi2}. Here `commutative function space' means that the semigroups of operators under investigation are defined on some $L_\infty(\Omega)$.  Note that  $L_\infty(\Omega)$ is the prototype of a  commutative von Neumann algebra. Even in this commutative setting a general theory of BMO spaces defined intrinsically
by the semigroup is far from established.

On the other hand, BMO spaces have been extended to noncommutative
function  spaces (i.e. von Neumann algebras) in various cases. Let
us refer to the seminal work on martingales in \cite{PX}, \cite{JX},
\cite{JM}, \cite{Mu} and \cite{PopaN}, and to \cite{NPTV}, \cite
{Mei1} and \cite{BP} for work on operator- or matrix-valued
functions. In   \cite{Mei} a first approach towards a  $H_1-BMO$
duality associated with semigroups of operators on von Neumann
algebras has been  obtained, whereas  a duality theory  for
Averson's subdiagonal algebras is studied in \cite{MW}.

As in the commutative case, BMO boundedness and interpolation
usually gives optimal or at least very good estimates for singular
integral operators on $L_p$. The use of BMO spaces also turns out to
be crucial when reducing results on group von Neumann algebras to
the semicommutative setting, see \cite{JMP}. Let us describe one of
our main results. Let $\N$ be a von Neumann algebra with a normal
trace $\tau$ satisfying $\tau(1)=1$, i.e. $(\N,\tau)$ is a
noncommuative probability space. Let $(T_t)$ be semigroup of
completely positive maps on $\N$ such that $\tau(T_t(x))=\tau(x)$
and $T_t(1)=1$. Then we define the $BMO_c$ column norm by
 \[ \|x\|_{BMO_c(\T)} \lel \sup_{t} \|T_t|x-T_tx|^2\|^{1/2} \]
and $\|x\|_{BMO(\T)}=\max\{
\|x\|_{BMO_c(\T)},\|x^*\|_{BMO_c(\T)}\}$. The norm
$\|x\|_{BMO_r(\T)}=\|x^*\|_{BMO_c(\T)}$ is called the row $BMO$ norm
and the need of both such norms is well-known from martingale
theory.

\begin{theorem} Assume that $T_t$ is a standard semigroup of completely positive maps on $\N$ and $(T_t)$ admits a Markov dilation. Then
 \[ [BMO(\T),L_1(\N)]_{\frac{1}{p}} \lel L_p(\N) \pl \] for
 $1<p<\infty$.
\end{theorem}

We investigate other possible intrinsic choices for $BMO$-norms and
compare them. These results are applied to BMO-boundedness of
Fourier multipliers on non-abelian discrete groups. We obtain their
corresponding $L_p$-boundedness with optimal constants. Basic
examples of Fourier multipliers in this article are noncommutative
analogues of E. Stein's imaginary power $(-\triangle)^{i\gamma}$
(see Theorem \ref{stein1} in Example \ref{stein}) and noncommutative
analogues of P. A. Meyer's generalized Riesz transforms (see Theorem
\ref{meyer}). A further application of our results gives optimal
constants in Junge/Xu's noncommutative maximal ergodic inequality
(see \cite{JX3}). Many of our results are new even in the
commutative setting. In particular, our constants of the $L_p$
bounds of Stein's universal Fourier multipliers are better than
those obtained by Stein (\cite{St}) and Cowling(\cite{Cow}) (see
Remark \ref{better}).

\section{Preliminaries and notation}
\subsection{Noncommutative $L_p$ spaces.}
Let $\mathcal{N}$ be a von Neumann algebra equipped with a normal
semifinite faithful trace $\tau $. Let $\mathcal{S}_{+}$ be the set
of all positive $f\in \mathcal{N} $ such that $\tau
(\mathrm{supp}(f))<\infty $, where $\mathrm{supp}(x)$ denotes the
support of $f$, i.e. the least projection $e\in \mathcal{N}$
such that $ef=f$. Let $\mathcal{S}_{\mathcal{N}}$ be the linear span of $%
S_{+}$. Note that $\mathcal{S}_{\mathcal{N}}$ is an involutive
strongly dense ideal of $\mathcal{N}$. For $0<p<\infty $ define
\[
\Vert f\Vert _p=\big(\tau (|f|^p)\big)^{1/p}\,,\quad x\in \mathcal{S}_{%
\mathcal{N}},
\]
where $|f|=(f^{*}f)^{1/2}$, the modulus of $x$. One can check that
$\Vert \cdot \Vert _p$ is a norm or $p$-norm on
$\mathcal{S}_{\mathcal{N}}$ according to $p\geq 1$ or $p<1$. The
corresponding completion is the noncommutative $L_p$-space
associated with $(\mathcal{N},\tau )$ and is
denoted by $L_p(\mathcal{N})$. By convention, we set $L^\infty (\mathcal{N})=%
\mathcal{N}$ equipped with the operator norm $\|\cdot\|$. The elements of $L_p(\mathcal{N%
})$ can be also described as measurable operators with respect to $(\mathcal{%
N},\tau )$. We refer to \cite{PX-surv} for more information and for
more historical references on noncommutative $L_p$-spaces. In the
sequel, unless explicitly stated otherwise, $\mathcal{N}$ will
denote a semifinite von Neumann algebra and $\tau $ a normal
semifinite faithful trace on $\mathcal{N}$. We will simplify
$L_p({\mathcal N})$ as $L_p$ and the corresponding norms as
$\|\cdot\|_p$.

We say an operator $T$ on $\mathcal{N}$ is completely contractive if $%
T\otimes I_n$ is contractive on $\mathcal{N}\otimes M_n$ for each $n$. Here,
$M_n$ is the algebra of $n$ by $n$ matrices and $I_n$ is the identity
operator on $M_n$. We say an operator $T$ on $\mathcal{M}$ is completely
positive if $T\otimes I_n$ is positive on $\mathcal{N}\otimes M_n$ for each $%
n$. We will need the following Kadison-Schwarz inequality for unital
completely positive contraction $T$ on $L_p(\mathcal{N})$,
\begin{eqnarray}\label{KS}
|T(f)|^2\leq T(|f|^2),\ \ \ \ \forall f\in L_p(\mathcal{N}).
\end{eqnarray}

\subsection{Standard noncommutative semigroups}

Throughout this article we will assume that $(T_t)$ is a semigroup
of completely positive maps on a semifinite von Neumann algebra $\N$
satisfying the following \emph{standard assumptions} \
 \begin{enumerate}
 \item[i)] Every $T_t$ is a  normal completely positive maps on
 $\N$ such that $T_t(1)=1$;
  \item[ii)] Every $T_t$ is selfadjoint with respect to the trace
  $\tau$, i.e. $\tau(T_t(f)g)=\tau(fT_t(g))$;
  \item[iii)] The family $(T_t)$ is strongly continuous, i.e. $\lim_{t\rightarrow0} T_tf=f$ with respect to the strong topology in $\N$ for any $f\in \N$.
 \end{enumerate}

Let us note that (i) and (ii) imply that $\tau(T_tx)= \tau(x)$ for
all  $x$, so $T_t$'s are faithful and are contractive on $L_1(\N)$.
By interpolation, $T_t$'s extend to contractions on $L_p(\N), 1\leq
p<\infty$ and satisfy $\lim_{t\rightarrow0} T_tx=x$ in $L_p(\N)$ for
all $x\in L_p(\N)$. (see \cite{JX3} for details). Some of these
conditions can be weakened, but this is beyond the scope of this
article.

Let us recall that such  a semigroup admits an infinitesimal
(negative) generator $A$ given as $Af=\lim_{t\to 0}
t^{-1}(f-T_t(f))$ defined on $\dom(A)=\cup_{1\leq p\leq \infty}
\dom_p(A),$ where
 \[  \dom_p(A)\lel \{f\in L_p(\N); \lim_{t\to 0} t^{-1}(T_t(f)-f) {\rm\ converges\ in\ } L_p(\N)\} \pl .\]
It is easy to see that $\frac1s\int_0^sT_t(f)dt\in \dom_p(A)$ for
any $s>0, f\in L_p(\N)$, so $\dom_p(A)$ is dense in $L_p(\N)$.
Denote by $A_p$ the restriction of $A$ on $\dom_p(A)$. Under our
assumptions (i)-(iii), $A_2$ is a positive (unbounded) operator.
$A_pT_t=T_tA_p=-\frac{\partial T_t}{\partial t}$ extend to a (same)
bounded operator on $L_p(\N)$ for all $t>0,1\leq p\leq \infty$.
Therefore, $T_s(f)\in \dom_p(A)$ for any $f\in L_p(\N)$, $1\leq
p\leq \infty$.

For a standard semigroup $T_s$ (generated by $A$), we may consider the
\emph{subordinated Poisson semigroup} $\P = (P_t)_{t \ge 0}$ defined
by $P_t = \exp(-tA^\frac12)$. $(P_t)$ is again a  semigroup of
operators satisfying (i)-(iii) above. Note that $P_t$ satisfies
$(\partial_t^2-A)P_t=0$.  By functional
calculus and an elementary identity, each $P_t$ can be written as
(see e.g.  \cite{Stein}),
 \begin{equation}\label{ps-form}
  P_t \lel \frac 1{2\sqrt{\pi }}\int_0^\infty te^{-\frac{t^2}{4u}}u^{-\frac
 32}T_udu.
 \end{equation}
 The integral on the right hand side of the identity converges with respect to the
operator norm on $L_p(N)$ for $1\leq p\leq \infty$. Let us define
the gradient form $\Gamma$ associated with $T_t$,
 \[ 2\Gamma(f,g) \lel (A(f^{*})g)+f^{*}(A(g))-A(f^{*}g) ,\]
for $f,g$ with $f^*,g, f^*g\in \dom(A)$. For convenience, we assume
that there exists a $^*$-algebra $\A$ which is weak$^*$ dense in
$\N$ such that $T_s(\A)\subset\A\subset \dom(A)$. This assumption is
to guarantee that $\Gamma(T_sf,T_sg)$ make senses for $f,g\in \A$,
which is not easy to verify in general, although the other form
$T_t\Gamma(T_sf,T_sg)$ is what we need essentially in this article
and can be read as
$T_t((AT_sf^{*})T_sg)+T_t(T_sf^{*}(AT_sg))-AT_t(T_sf^{*}T_sg)$ for
any $f,g\in L_p(\N),1\leq p\leq\infty, s,t>0$. The semigroup $(T_t)$
generated by $A$ is said to satisfy the $\Gamma^2\geq0$ if
 \[ \Gamma (T_v f, T_vf) \le T_v \Gamma (f,f)\]
for all $v>0,f\in \A$. It is easy to see $\Gamma^2\gl 0$ also
implies $\Gamma(P_vf,P_vf)\le P_v\Gamma(f,f)$ for any $v>0$. Denote
by the gradient form associated with $(P_t)_t$ by
$\Gamma_{A^\frac12}$.

We will need the following Lemma proved in \cite{Jur} and
\cite{JMr}. We add a short proof for the convenience of the reader.
\begin{lemma}\label{lemma}
 (i) For any $f\in L_p(\N),1\leq p\leq\infty, s>0$, we have
 \[ T_s|f|^2-|T_sf|^2 \lel 2 \int_{0}^s  T_{s-t}\Gamma(T_tf,T_tf) dt \pl .\]

 (ii) For any $f\in \A$, we have
 \[ \Gamma_{A^\frac12}(f,f) \lel \int_0^{\infty}
 P_{v}\Gamma(P_{v}f,P_{v}f) dv +
 \int_0^{\infty}
 P_{v}|P_{v}'f|^2  dv \]
For any $f\in L_p(\N), s, t>0$, we have
 \[P_t\Gamma_{A^\frac12}(P_sf,P_sf) \lel \int_0^{\infty}
 P_{t+v}\Gamma(P_{s+v}f,P_{s+v}f) dv +
 \int_0^{\infty}
 P_{t+v}|P_{s+v}'f|^2  dv.\]

 Here and in the rest of the article, $f_t'$ means $\frac{df_t}{dt}$.
\end{lemma}

\begin{proof} (i): For $s$ fixed, let
\[
F_t=T_{s-t}(|T_tf|^2).
\]
Then
\begin{eqnarray*}
\frac{\partial T_{s-t}(|T_tf|^2)}{\partial t} &=&\frac{\partial T_{s-t}}{%
\partial t}(|T_tf|^2)+T_{s-t}[(\frac{\partial T_t}{\partial t}%
f^{*})f]+T_{s-t}[f^*(\frac{\partial T_t}{\partial t}f)] \\
&=&-T_{s-t}\Gamma (T_tf,T_tf).
\end{eqnarray*}
Therefore
\begin{eqnarray*}
T_s|f|^2-|T_sf|^2 =-F_s+F_0=\int_0^sT_{s-t}\Gamma (T_tf,T_tf)dt.
\end{eqnarray*}

(ii): Let
\[
F_t=\frac{\partial P_t}{\partial t}(|P_tf|^2)-P_t(\frac{\partial P_tf^{*}}{%
\partial t}P_tf)-P_t(P_tf^{*}\frac{\partial P_tf}{\partial t}).
\]
Then
\begin{eqnarray*}
\frac{\partial F_t}{\partial t} &=&\frac{\partial ^2P_t}{\partial t^2}%
(|P_tf|^2)-P_t(\frac{\partial ^2P_tf^{*}}{\partial t^2}P_tf)-P_t(P_tf^{*}%
\frac{\partial ^2P_tf}{\partial t^2})-2P_t(|\frac{\partial P_tf}{\partial t}%
|^2) \\
&=&-AP_t(|P_tf|^2)-P_t[(-AP_tf^{*})P_tf]-P_t[P_tf^{*}(-AP_tf)]-2P_t(|\frac{%
\partial P_tf}{\partial t}|^2) \\
&=&-P_t\Gamma (P_tf,P_tf)-2P_t(|P_t^{\prime }f|^2).
\end{eqnarray*}
Note that $F_0=\Gamma _{A^\frac 12}(f,f)$ and $F_t\rightarrow 0$ in ${\mathcal N}$ as $%
t\rightarrow \infty$ because of Proposition 1.1. We get
\begin{align*}
\Gamma _{A^\frac 12}(f,f) &=\int_0^\infty -\frac{\partial
F_t}{\partial t}dt  \lel \int_0^\infty P_t\Gamma
(P_tf,P_tf)dt+2\int_0^\infty P_t(|P_t^{\prime }f|^2)dt. \qedhere
\end{align*}
\end{proof}

We will use the following inequality from
\cite{Mei}.

\begin{prop}\label{monot} Let $f\in \N$ be positive  and $0<t<s$. Then
 \[ P_sf\kl  \frac{s}{t} \pl  P_tf  \pl .\]
\end{prop}

\begin{proof} We use \eqref{ps-form} and $e^{-\frac{s^2}{4u}}\le
e^{-\frac{t^2}{4u}}$ for all $u$. This yields the assertion
\begin{align*}
 \frac{P_sf}s &=\frac 1{2\sqrt{\pi }}\int_0^\infty e^{-\frac{s^2}{4u}}u^{-\frac 32}T_u(f)du  \kl  \frac 1{2\sqrt{\pi }}\int_0^\infty
 e^{-\frac{t^2}{4u}}u^{-\frac 32}T_u(f)du \lel
 \frac{P_tf}t \pl . \qedhere
\end{align*}\end{proof}

\section{BMO norms associated with semigroups of operators}
 In
this part we study several natural BMO-norms associated with a
semigroup $T_t$ of completely positive maps. The situation is
particularly nice for subordinated semigroups so that the original
semigroup satisfies the $\Gamma^2\geq0$.

Given  a standard semigroup of operators $T_t$ on $\N$ and $f\in
\N\cup L_2(\N)$, we define
 \begin{eqnarray}
  \|f\|_{bmo^c(\T)} &\lel& \sup_t \|
 T_t|f|^2-|T_tf|^2\|^{\frac12}\pl, \\
  \|f\|_{BMO^c(\T)} &\lel& \sup_t \|
 T_t|f-T_tf|^2\|^{\frac12}\pl.
 \end{eqnarray}
Here and in what follows $\|f\|$ always denote the operator norm
of $f$. The notations  $\|\cdot\|_{bmo^c(\P)},
\|\cdot\|_{BMO^c(\P)}$ are used when $\T$ is replaced by the
 subordinated semigroup $(P_t)$ above. The definitions steam from
Garsia's norm for the Poisson semigroup on the circle (see
\cite{Koo}). The $BMO^c(\T)$-norm has been studied in \cite{Mei},
motivated by the expression
 \[ \|f\|_{BMO_1} \lel  \sup_z P_z(|f-f(z)|)  \pl .\]
This definition appeared in the commutative case in particular in
\cite{SV} , \cite{Vo} and \cite{DuLi2}. Using the $\| \pl
\|_{BMO_1}$-norm, it is easy to show that the conjugation operator
is bounded from $L_{\infty}$ to BMO$_1$. Here $f(z)$ gives the value
of the harmonic extension in the interior of the circle (see
\cite{Ga}). In some sense $f-P_tf$ is similar to $f-f(z)$, despite
the fact that the Poisson integral $P_tf$ still is a function, while
$f(z)$ is considered as a constant function in $f-f(z)$.

\begin{prop}  Let  $(T_t)$ be a standard semigroup of operators. Then
$bmo^c(\T)$ and $BMO^c(\T)$ are semi-norms on $\N$.
\end{prop}

\begin{proof} Fix $t>0$.  Let
${\mathcal L}(\N\ten_{T_t}\N)$ be the Hilbert $C^*$-module over $\N$
with $\N$-valued inner product
 \[ \langle a\ten b,c\ten d\rangle \lel b^*T_t(a^*c)d \pl. \]
This Hilbert $C^*$-module is well-known from the GNS-construction
for $T_t$, see \cite{La}. Since $T_t$ is unital, we have a
$^*$-homomorphism $\pi:{\mathcal N}\to {\mathcal L}({\mathcal
N}\ten_{ T_t}\N)$ such that
 \[ T_t(f) \lel e_{11}\pi(f)e_{11} \pl .\]
We then get
  \begin{align*}  & T_t(f^*f)-T_t(f^*)T_t(f)
  \lel e_{11}\pi(f)^*\pi(f) e_{11}-
e_{11}\pi(f)^*e_{11}e_{11}\pi(f)
 e_{11}\\
 &= e_{11}\pi(f)^*(1-e_{11})\pi(f)e_{11}
 \pl .\end{align*}
Therefore,
\begin{eqnarray*}
\|T_t|f|^2-|T_tf|^2\|^{\frac12}=\|(1-e_{11})^{\frac12}
\pi(f)e_{11}\|_{{\mathcal L}({\mathcal
N}\ten_{ T_t}\N)},\\
\|T_t|f-T_tf|^2\|^{\frac12}=\|\pi(f-e_{11}\pi(f)e_{11})e_{11}\|_{{\mathcal
L}({\mathcal N}\ten_{ T_t}\N)}.
\end{eqnarray*}
This shows that $\|\cdot\|_{bmo^c(\T)}$ and
$\|\cdot\|_{BMO^c(\T)}$ are semi-norms.\end{proof}
\begin{rem} {\rm
An alternative proof for $bmo^c(\T)$ being a semi-norm can be
derived from the identity of Lemma \ref{lemma} (i). Using the GNS
construction for the positive form $T_{t-s}\Gamma$ we can find
linear maps $u_{ts}:{\mathcal N}\to C({\mathcal N})$ such that
 \[ T_t|f|^2-|T_tf|^2 \lel \int_0^t |u_{ts}(f)|^2 ds \pl.\]
This provides an embedding in $L_2^c([0,t])\ten_{\min}C({\mathcal
N})$.}
\end{rem}

 \begin{prop}\label{kernel} Let $(T_t)$ be a standard semigroup
 and $f\in \N\cup L_2(\N)$.
Then the following conditions are equivalent:
\begin{enumerate}
\item[(i)]$\|f\|_{bmo^c(\T)}=0$. \item[(ii)]
 $\|f\|_{BMO^c(\T)}=0$.
\item[(iii)] $f\in ker(A_\infty)\cup ker(A_2)=\{f\in \dom_\infty(A)\cup \dom_2(A), Af=0\}$.
\end{enumerate}
\end{prop}

\begin{proof}
Note that  (ii) and (iii) both equals to $T_tf=f$ for any $t$ since
$T_t$ is faithful. Hence (ii) is equivalent to (iii). Assume (iii),
then $\tau(T_t|f|^2-|T_tf|^2)=0$ since $T_tf=f$ for all $t$ and
$T_t$ is trace preserving. Note $T_t|f|^2-|T_tf|^2\geq0$ by
(\ref{KS}), so $T_t|f|^2-|T_tf|^2=0$ for any $t>0$. We get (i).
Assume (i), we have $\tau (T_t|f|^2-|T_{t}f|^2) =0$, so $\tau
(|f|^2-|T_tf|^2)=0$ for any $t>0$. So $\tau(|f-T_{t}f|^2)=\tau
(|f|^2-2|T_{t}f|^2+|T_{2t}f|^2)=0$ for any $t>0$. So $f=T_tf$ for
any $t>0$. This implies (iii).
\end{proof}

\begin{prop}\label{gp} Let $(T_t)$ be a standard semigroup and  $f\in {\mathcal
N}\cup L_2(\N)$. Then
\begin{enumerate}
 \item[(i)] $\|T_sf\|_{bmo^c(\T)}\kl \|f\|_{bmo^c(\T)}$ for all $s>0$;
  \item[(ii)] $\|f\|_{BMO^c(\T)}\kl 2\pl \|f\|_{bmo^c(\T)}+\sup_t
 \|T_tf-T_{2t}f\|$.
 \item[(iii)] If in addition  $\Gamma^2\geq0$, then
 $$\|f\|_{BMO^c(\T)}\simeq \|f\|_{bmo^c(\T)}+\sup_t
 \|T_tf-T_{2t}f\|.$$ \end{enumerate}
\end{prop}

\begin{proof} Let us start with i) and the pointwise estimate
  \[ 0\kl T_{t}|T_sf|^2-|T_{t+s}f|^2\kl T_{t+s}|f|^2-|T_{t+s}f|^2
  \pl .\]
By definition of the $bmo^c(\T)$ seminorm  this implies
 \begin{align*}
  \|T_sf\|_{bmo^c(\T)}
 &=  \sup_t \| T_{t}|T_sf|^2-|T_{t+s}f|^2\|^{\frac12} \kl
  \sup_t \|T_{t+s}|f|^2-|T_{t+s}f|^2\|^{\frac12} \kl
 \|f\|_{bmo^c(\T)} \pl .
 \end{align*}
For the proof of (ii), we fix  $t>0$ and use the triangle inequality
(see Lemma \ref{bmo1}):
\begin{align*}
 \|T_{t}|f-T_tf|^2\|^{\frac12}
 &\kl \|T_{t}|f-T_tf|^2-|T_t(f-T_tf)|^2\|^{\frac12} +
  \|\p  |T_t(f-T_tf)|^2\|^{\frac12} \\
  &\le \|f-T_tf\|_{bmo^c(\T)}+\||T_t(f-T_tf)|^2\|^{\frac12}\\
  &\le  \|f\|_{bmo^c(\T)}+\|T_tf\|_{bmo^c(\T)}+\|T_t(f-T_tf)\|
 \end{align*}
We apply (i) and obtain
\begin{align*}
  \|T_{t}|f-T_tf|^2\|^{\frac12}
 &\le  2\|f\|_{bmo^c(\T)}+\|T_t(f-T_tf)\|.
 \end{align*}
Taking supremum over $t$ yields the assertion. To prove (iii), we
apply Lemma \ref{lemma} (i) and the triangle inequality,
\begin{eqnarray*}
(T_t|f|^2-|T_tf|^2)^{\frac12}&=&(\int_0^tT_{t-s}\Gamma(T_sf,T_sf)ds)^{\frac12}\\
&\leq&(\int_0^tT_{t-s}\Gamma(T_s(f-T_tf),T_s(f-T_tf))ds)^{\frac12}+(\int_0^tT_{t-s}\Gamma(T_{s+t}f,T_{s+t}f)ds)^{\frac12}\\
({\rm Lemma\ \ref{lemma}\ (i)})&=&(T_t|f-T_tf|^2-|T_tf-T_{2t}f|^2)^{\frac12}+(\int_0^tT_{t-s}\Gamma(T_{s+t}f,T_{s+t}f)ds)^{\frac12}\\
(\Gamma^2\geq0)&\leq&(T_t|f-T_tf|^2)^{\frac12}+(\int_0^tT_{2t-2s}\Gamma(T_{2s}f,T_{2s}f)ds)^{\frac12}\\
(v=2s)&\leq&(T_t|f-T_tf|^2)^{\frac12}+(\frac12\int_0^{2t}T_{2t-v}\Gamma(T_{v}f,T_{v}f)dv)^{\frac12}\\
({\rm Lemma\ \ref{lemma}\
(i)})&=&(T_t|f-T_tf|^2)^{\frac12}+\frac1{\sqrt2}(T_{2t}|f|^2-|T_{2t}f|^2)^{\frac12}.
\end{eqnarray*}
Taking the norm and the supremum over $t$ on both sides, we get
$$\|f\|_{bmo^c(\T)}\leq (\sqrt2+2)\|f\|_{BMO^c(\T)}.$$
By Choi's inequality (see \cite{Choi}) we find
 \[ |T_tf-T_{2t}f|^2 \kl T_t|f-T_tf|^2 \pl . \]
Together with (ii), we obtain  (iii). \qd

We now consider BMO-norms associated with the subordinated semigroup
$(P_t)_t$.
\begin{prop}\label{bmo3} Let $(T_t)$ be a standard semigroup and $(P_t)$ be the associated Poisson semigroup. Let
$f\in \N\cup L_2(\N)$. Then
\begin{enumerate}
 \item[(i)] $P_b|f|^2-|P_bf|^2=   2 \int_{0}^{\infty} \int_{\max\{0,v-b\}}^{v}
  P_{b-v+2t}\hat{\Gamma}(P_{v}f,P_{v}f) dt dv$;
 \item[(ii)] $\sup_b \|\int_0^{\infty} P_{b+s}|P'_sf|^2 \min(s,b)
 ds\| \kl 4 \pl \|f\|_{bmo^c(\P)}^2$;
 \item[(iii)] If in addition  $\Gamma^2\gl 0$, then
 \[ \frac14 \int_0^{\infty}  P_{b+s}\hat{\Gamma}(P_{s}f,P_sf)   \min(s,b)
 ds \kl P_b|f|^2-|P_bf|^2\kl 180 \int_0^{\infty} P_{\frac{b}{3}+s}
 \hat{\Gamma}(P_{s}f,P_{s}f) \pl \min(\frac b3,s) ds\pl.\]
\end{enumerate}
Here $\hat{\Gamma}(f_s,f_s)=\Gamma(f_s,f_s)+|f_s'|^2$.
\end{prop}

\begin{proof} For the proof of (i) we apply Lemma \ref{lemma} (i) to $P_t$ and get
 \[ P_b|f|^2-|P_bf|^2 \lel 2 \int_{0}^b  P_{b-s}\Gamma_{A^{\frac12}}(P_sf,P_sf) ds \pl .\]
Using the formula for $\Gamma_{A^{\frac12}}(P_sf,P_sf) $ from Lemma
\ref{lemma} (ii), we obtain with the  change of variable ($v=s+t$)
that
 \begin{align}
  P_b|f|^2-|P_bf|^2 &\lel  2 \int_{0}^b \int_0^{\infty}
  P_{b-s+t}\hat{\Gamma}(P_{s+t}f,P_{s+t}f) dtds  \nonumber \\
  &\lel
  2 \int_{0}^{\infty} \int_{t}^{b+t}
  P_{b-v+2t}\hat{\Gamma}(P_{v}f,P_{v}f) dv dt\nonumber\\
&\lel
  2 \int_{0}^{\infty} \int_{\max\{0,v-b\}}^{v}
  P_{b-v+2t}\hat{\Gamma}(P_{v}f,P_{v}f) dt dv.
  \label{start}
\end{align}
This is (i). Note $\Gamma^2\geq0$ implies $\hat{\Gamma}^2\gl0$. We
apply monotonicity, Proposition \ref{monot} and split the
integral, and get
 \begin{align*}
  & 2 \int_{0}^{\infty} \int_{\max\{0,v-b\}}^{v}
  P_{b-v+2u}\hat{\Gamma}(P_{v}f,P_{v}f) du dv\\
  &\gl 2 \int_{0}^{\infty} \int_{\max\{0,v-b\}}^{v}
  P_{b+2u}\hat{\Gamma}(P_{2v}f,P_{2v}f) du dv\\
  &\gl 2 \int_{0}^{\infty} \kla \int_{\max\{0,v-b\}}^v \frac{b+2u}{b+2v}
   du\mer  P_{b+2v}\hat{\Gamma}(P_{2v}f,P_{2v}f)  dv\\
 &\lel   \int_{0}^{\infty}  \frac{2bv+4v^2-2b\max\{0,v-b\}-4\max\{0,v-b\}^2}{2(b+2v)}
   P_{b+2v}\hat{\Gamma}(P_{2v}f,P_{2v}f)  dv\\
  &\gl   \int_{0}^{b}
   P_{b+2v}\hat{\Gamma}(P_{2v}f,P_{2v}f)  vdv
  +  \int_{b}^{\infty} \frac{4bv}{2(b+2v)}
    P_{b+2v}\hat{\Gamma}(P_{2v}f,P_{2v}f)  dv
    \\
  &\gl \frac12  \int_{0}^{b} P_{b+2v} \hat{\Gamma}(P_{2v}f,P_{2v}f)  2vdv
  +\frac{1}{2} b  \int_{b}^{\infty} P_{b+2v}\hat{\Gamma}(P_{2v}f,P_{2v}f)
  dv \\
  &\gl \frac12 \int_0^{\infty} P_{b+2v}\hat{\Gamma}(P_{2v}f,P_{2v}f)  \min(2v,b)
  dv\pl .
 \end{align*}
Without $\Gamma^2\gl 0$ we  only obtain
 \[ P_b|f|^2-|P_bf|^2 \gl
 \frac12 \int_0^{\infty}  P_{b+2v} |P'_{2v}f|^2   \min(2v,b) dv
 \lel \frac14 \int_0^{\infty}  P_{b+v} |P'_{v}f|^2   \min(v,b) dv
 \pl .\]
This is (ii) and the first inequality of (iii). To complete the
proof of (iii) we start with \eqref{start} and  $\Gamma^2\geq0$:
 \begin{align*}
  &P_b|f|^2-|P_bf|^2 \lel 2 \int_{0}^{\infty} \int_{\max\{0,v-b\}}^{v}
  P_{b-v+2u}\hat{\Gamma}(P_{v}f,P_{v}f) du dv \\
 &\kl 2 \int_{0}^{\infty} \int_{\max\{0,v-b\}}^{v}
  P_{b-\frac{v}{3}+2u}\hat{\Gamma}(P_{\frac{v}{3}}f,P_{\frac{v}{3}}f) du dv \\
 &\lel  2 \int_0^b \int_{0}^{v}  P_{b-v+2u}\hat{\Gamma}(P_{v}f,P_{v}f) du dv
 + 2 \int_b^{\infty} \int_{v-b}^{v}  P_{b-v+2u}\hat{\Gamma}(P_{v}f,P_{v}f) du dv
 \lel I+II \pl .
 \end{align*}
For $v\gl b$ we have
 \[  \frac{b+v}{3} \kl  b-\frac{v}{3}+2u \kl \frac{5}{3}(b+v) \pl. \]
Thus monotonicity  implies
 \begin{align*}
  II &\kl 2 \int_{0}^{\infty} \int_{v-b}^{v}
  P_{b-\frac{v}{3}+2u}\hat{\Gamma}(P_{\frac{v}{3}}f,P_{\frac{v}{3}} f) du dv
 \kl   10b  \int_{b}^{\infty}
 P_{\frac{b+v}{3}}\hat{\Gamma}(P_{\frac{v}{3}}f,P_{\frac{v}{3}}f) dv \\
 &= 90   \int_{\frac b3}^{\infty}
 P_{\frac{b}{3}+s}\hat{\Gamma}(P_{s}f,P_sf)
 \min(s,\frac{b}{3}) \pl
  ds\pl .
 \end{align*}
In the range $v\le b$ and $0\le u\le v$  we also have
 \[ \frac{b+v}{3}\kl b+2u-\frac{v}{3} \kl \frac{5}{3}(b+v) \pl .\]
Again by monotonicity  and with $s=\frac{v}{3}$ we obtain
 \begin{align*}
 I &\le 10  \int_0^{b}  P_{\frac{b+v}{3}}\hat{\Gamma}(P_{\frac{v}{3}}f,P_{\frac{v}{3}}
 f) \pl v dv
 \lel 90 \int_0^{\frac b3}  P_{\frac{b}{3}+s}\hat{\Gamma}(P_{s}f,P_{s}
 f) \pl s  ds \pl.
 \end{align*}
This yields
 \begin{align*}
  P_b|f|^2-|P_bf|^2 &\le  180 \int_0^{\infty} P_{\frac{b}{3}+s}
 \hat{\Gamma}(P_{s}f,P_{s}f) \pl \min(\frac{b}{3},s) ds \pl .\qedhere
 \end{align*}\qd

In view of the classical Carleson-measure-characterization of BMO,
we define, for $f\in \N\cup L_2(\N)$,
\begin{eqnarray}
\|f\|_{BMO^c(\partial)}&=&\|\sup_t P_t\int_0^t
|P_s'f|^2sds\|^{\frac12},\\
\|f\|_{BMO^c(\Gamma)}&=&\sup_t
   \|P_t\int_0^{t}\!\!
  \Gamma(P_{s}f,P_{s}f) s ds\|^{\frac 12 },\\
  \|f\|_{BMO^c(\hat{\Gamma})} &=&
 \sup_t \|P_t\int_0^{t} \hat{\Gamma}(P_sf,P_sf)
 sds\|^{\frac12}  \pl .
 \end{eqnarray}

\begin{theorem}\label{lem0} Let $(T_t)$ be a standard semigroup. Then
\begin{eqnarray}
\|f\|_{BMO^c(\partial)}\leq c\|f\|_{BMO^c(\P)}\leq
c\|f\|_{BMO^c(\T)}
\end{eqnarray}
\end{theorem}
\begin{proof}
To prove the first inequality, recall that Theorem 3.2 of \cite{Mei}
states that \begin{eqnarray*} \sup_t\|P_t\int_0^t|\frac{\partial
P_s}{\partial s}(f-P_sf)|^2sds\|^{\frac12}\leq \|f\|_{BMO^c(\P)}.
\end{eqnarray*}
Then
\begin{eqnarray*} &&\|P_t\int_0^t|\frac{\partial
P_s}{\partial s}f|^2sds\|^{\frac12}\\
&\leq& \|P_t\int_0^t|\frac{\partial P_s}{\partial
s}(f-P_sf)|^2sds\|^{\frac12}+\|P_t\int_0^t|\frac{\partial
P_s}{\partial
s}P_sf|^2sds\|^{\frac12}\\
&\leq& \|f\|_{BMO^c(\P)}+\|P_t\int_0^{\frac t2}|\frac{\partial
P_s}{\partial s}P_sf|^2sds+P_t\int_{\frac t2}^t|\frac{\partial
P_s}{\partial s}P_sf|^2sds\|^{\frac12}\\
(v=2s)&\leq& \|f\|_{BMO^c(\P)}+\|\frac14P_t\int_0^{\frac
t2}|\frac{\partial P_v}{\partial v}f|^2vdv+P_t\int_{\frac
t2}^tP_{\frac t2}|\frac{\partial
P_s}{\partial s}P_{s-\frac t2}f|^2sds\|^{\frac12}\\
(u=2s-\frac t2) &\leq& \|f\|_{BMO^c(\P)}+\|\frac14P_t\int_0^{\frac
t2}|P_v'f|^2vdv+\frac12P_{\frac {3t}2}\int_{\frac
t2}^{\frac {3t}2}|P_u'f|^2udu\|^{\frac12}\\
&\leq& \|f\|_{BMO^c(\P)}+\frac{\sqrt3}2 \sup_t\|P_t\int_0^{
t}|P_s'f|^2sds\|^{\frac12}.
\end{eqnarray*}
Taking supremum on both sides, we have, \begin{eqnarray*}
\|f\|_{BMO(\partial)}=\sup_t\|P_t\int_0^{t}|\frac{\partial
{P_s}}{\partial s}f|^2sds\|^{\frac12}\leq
\frac1{1-\frac{\sqrt3}2}\|f\|_{BMO^c(\P)}.
\end{eqnarray*}
To prove the second inequality, we apply (\ref{ps-form}) and
(\ref{KS}),
\begin{eqnarray*}
P_s|f-P_sf|^2 &=&\int_0^\infty \phi _s(u)T_u|f-\int_0^\infty \phi
_s(v)T_vfdv|^2du \\
&=&\int_0^\infty \phi _s(u)T_u|\int_0^\infty \phi _s(v)(f-T_vf)dv|^2du \\
&\leq &\int_0^\infty \int_0^\infty \phi _s(u)\phi
_s(v)T_u|f-T_vf|^2dvdu
\end{eqnarray*}
with $\phi _s(v)=se^{\frac{-s^2}{4v}}v^{\frac{-3}2}.$ For $v\leq u,$
we have
\[
\|T_u|f-T_vf|^2\|=\|T_{u-v}T_v|f-T_vf|^2\|\leq \|f\|^2_{BMO^c(\T)}.
\]
For $v>u,$ let $k$ be the biggest integer smaller than $\log
_2^{\frac vu},$ we have
\begin{align*}
&\|T_u|f-T_vf|^2\|^{\frac 12} \\
&\leq \|T_u|f-T_uf|^2\|^{\frac 12}+\|T_u|T_uf-T_{2u}f|^2\|^{\frac
12}+\|T_u|T_{2u}f-T_{4u}f|^2\|^{\frac
12}+\cdots+\|T_u|T_{2^ku}f-T_vf|^2\|^{\frac 12}
\\
&\leq  c(\ln \frac vu+1)\|f\|_{BMO^c(\T)}.
\end{align*}
Therefore, we find that
\begin{eqnarray*}
\|P_s|f-P_sf|^2\|^\frac12 &\leq& (\int_0^\infty \int_u^\infty \ln
\frac vu\phi _s(u)\phi _s(v)dvdu+\int_0^\infty \int_0^\infty \phi
_s(u)\phi _s(v)dvdu)\|f\|_{BMO^c(\T)}\\&\leq& c\|f\|_{BMO^c(\T)}.
\end{eqnarray*}
Taking supremum over $t$ yields the second inequality.
\end{proof}

\begin{lemma}\label{lem1} Let $(T_t)$ be a standard semigroup. Then
\begin{eqnarray}
\|f\|_{BMO^c(\partial)}&\simeq&\sup_t \|\int_0^\infty
P_{s+t}|P_s'f|^2\min(s,t) ds\|^\frac12.\label{lem1partial}\\
{\text{ If in addition }} \Gamma^2\geq0,& &\nonumber\\
\|f\|_{BMO^c(\Gamma)}&\simeq& \|\sup_t\int_0^\infty
P_{s+t}\Gamma(P_sf,P_sf)\min(s,t) ds\|^\frac12.\label{lem1Gamma}\\
\|f\|_{BMO^c(\hat \Gamma)}&\simeq& \|\sup_t\int_0^\infty P_{s+t}\hat{\Gamma }%
 (P_sf,P_sf)\min(s,t) ds\| \label{lem1hat}.
 \end{eqnarray}
\end{lemma}

\begin{proof}  Let $\partial_t=\frac \partial{\partial t}$ and $\Gamma_{\partial_t^2}(f,f)=|\frac{\partial
f_t}{\partial t}|^2$ the gradient forms associated with
$T_t=e^{t\partial_t^2}$ which satisfies
$\Gamma^2_{\partial_t^2}\geq0$. Then (\ref{lem1partial}) follows
from (\ref{lem1Gamma}). Moreover, (\ref{lem1hat}) follows from
$\hat{\Gamma}(f_t,f_t)=\Gamma(f_t,f_t)+|f_t'|^2$. To prove
(\ref{lem1Gamma}), we apply the condition $\Gamma^2\geq0$ and find
 \begin{align*}
  &\|\int_0^tP_t{\Gamma }(P_vf,P_vf)vdv \|
   \kl \|\int_0^tP_{\frac v2+t}{\Gamma }(P_{\frac v2}f,P_{\frac
   v2}f)vdv\|\\
  &=  4 \pl \|\int_0^{\frac t2}P_{s+t}{\Gamma}(P_sf,P_sf)sds \|
  \kl  4 \pl \|\int_0^\infty P_{s+t}{\Gamma }(P_sf,P_sf) \min(s,t) ds \| .
 \end{align*}
For the reversed relation, we use a dyadic decomposition. Indeed,
according to Proposition \ref{monot}, we have
 \[
 \frac{2^ntP_{s+t}}{s+t}{\Gamma }(P_sf,P_sf)\leq P_{2^nt}{%
 \Gamma }(P_sf,P_sf)
 \]
for $s\geq 2^nt.$ This implies
  \begin{align*}
& \frac12 \int_0^\infty P_{s+t}{\Gamma }(P_sf,P_sf) \min(s,t) ds\\
 &\kl
 \int_0^\infty P_{s+t}{\Gamma }(P_sf,P_sf) \frac{st}{s+t} ds  \\
&\lel \int_0^{2t}P_t{\Gamma}(P_sf,P_sf)\frac{st}{s+t}ds+\sum_{n=1}^\infty \frac1{2^n}\int_{2^nt}^{2^{n+1}t}\frac{2^ntP_{s+t}}{s+t}{\Gamma }(P_sf,P_sf)sds \\
 &\kl \int_0^{2t}P_t{\Gamma}(P_sf,P_sf)sds+\sum_{n=1}^\infty \frac1{2^n}\int_{2^nt}^{2^{n+1}t}P_{2^nt}{\Gamma }(P_sf,P_sf)sds \\
 &\kl   \int_0^{2t}P_t{\Gamma}(P_sf,P_sf)sds+\sum_{n=1}^\infty
\frac1{2^n}\int_0^{2^{n+1}t}P_{2^nt}{\Gamma }(P_sf,P_sf)sds.
 \end{align*}
However, we can replace $2t$ by $t$ using $\Gamma^2\gl 0$ and Lemma
\ref{monot}:
 \begin{align*}
  &\int_0^{2t}P_t{\Gamma }(P_sf,P_sf)sds
  \kl \int_0^{2t}P_{t+\frac s2}  {\Gamma }(P_{\frac s2}f,P_{\frac s2} f)sds
  \lel 4  \int_0^{t} P_{t+v}  {\Gamma }(P_vf,P_vf) vdv\\
  &\le  8 \int_0^{t} P_{t}  {\Gamma }(P_vf,P_vf) vdv
  \pl.\end{align*}
Applying this argument for every $2^{n+1}t$, we deduce the
assertion. \qd


\begin{lemma}\label{aa} Let $a>1$. Then
 \[
   \sup_t \|P_tf-P_{at}f\| \kl  \sqrt{2}(1+{\rm \log }_{\frac 32}a)
  \|f\|_{BMO^c(\partial)}.
\]
\end{lemma}

\begin{proof} For $t$ fixed, we have the
 \begin{align*}
 &|P_{3t}f-P_{2t}f|^2 \kl
  P_{\frac{3t}2}(|P_{\frac{3t}2}f-P_{\frac t2}f|^2)
  \lel P_{\frac{3t}2}(|\int_{\frac t2}^{\frac{3t}2} P_s'f ds|^2)
 \\
 & \kl   P_{\frac{3t}2}(t\int_{\frac t2}^{\frac{3t}2}|P'_sf|^2ds)
  \kl  2P_{\frac{3t}2}(\int_{\frac t2}^{\frac{3t}2}|P'_sf|^2sds)
 \kl
  2P_{\frac{3t}2}(\int_0^{\frac{3t}2}|P'_sf|^2sds).
\end{align*}
This implies in particular that
 \[ \sup_t \| P_tf-P_{\frac{3t}2}f\| \kl
 \sqrt{2} \|f\|_{BMO^c(\partial)}.
 \]
For  $1<a\leq \frac 32$, choose $b\geq 0$ such that
$\frac{a-b}{1-b}=\frac 32$. Then we obtain
  \begin{align}
  \|\p  |P_tf-P_{at}f|^2 \|  &\le \| P_{bt}|P_{(1-b)t}(f)-P_{\frac
 32(1-b)t}(f)|^2 \|   \nonumber \\
   &\le  \| \p |P_{(1-b)t}(f)-P_{\frac 32(1-b)t}(f)|^2 \|
   \kl 2 \pl \|f\|_{BMO^c(\partial)}^2 \pl .
 \end{align}
We deduce
 \begin{equation}
 \|P_t(f)-P_{at}(f) \| \kl  \sqrt{2} \pl \|f\|_{BMO^c(\partial)} \label{sqrt2}
 \end{equation}
for any $1<a\le \frac32$. Consider now $a>\frac 32$. Let  $n$ be the
integer part of $\log _{\frac 32}a$.  We may use a telescopic sum
 \[ P_tf-P_{at}f \lel (P_tf-P_{\frac{3t}2}f)+(P_{\frac{3t}2}f-P_{_{\frac 32\frac{3t}2}}f)+\cdots (P_{(\frac
 32)^nt}f-P_{at}f)\pl .
 \]
We apply \eqref{sqrt2} for every summand. Then the triangle
inequality implies the assertion. \qd

\begin{theorem}\label{last} Let $(T_t)$ be a standard semigroup satisfying $\Gamma^2\gl 0$. Then $\|\pl
\|_{BMO^c(\P)}$, $\|\pl \|_{bmo^c(\P)}$ and $\|\pl
\|_{BMO^c(\hat{\Gamma})}$ are all equivalent on $\N\cup L_2(\N)$.
\end{theorem}

\begin{proof} According to Proposition \ref{bmo3} we know that
 \[  \sup_t \|\int_0^\infty P_{s+t}\hat{\Gamma }(P_sf,P_sf)\min(s,t)ds\|^{\frac 12}
 \sim_{180} \|f\|_{bmo^c(\P)} \pl .\]
Then Lemma \ref{lem1} implies that $\|\pl \|_{bmo^c(\P)}$ and
$\|\pl\|_{BMO^c(\hat{\Gamma})}$ are equivalent. Proposition \ref{gp}
(iii) provides the upper estimate of $\|\pl \|_{bmo^c(\P)}$ against
$\|\pl\|_{BMO^c(\P)}$. Conversely, we deduce from Proposition
\ref{gp} (ii), Lemma \ref{aa}, Lemma \ref{lem1} and  Proposition
\ref{bmo3} (i) that
 \begin{align*}
 \|f\|_{BMO^c(\P)}&\kl 2 \|f\|_{bmo^c(\P)}+ \sup_t
 \|P_tf-P_{2t}f\| \\
 &\le 2 \|f\|_{bmo^c(\P)} + \sqrt{2}(1+\log_{\frac32}2)
 \|f\|_{BMO^c(\hat{\Gamma})} \\
 &\le 2 \|f\|_{bmo^c(\P)} + 2\sqrt{2}\pl 2\sqrt{6}\pl \|f\|_{bmo^c(\P)}
 \lel (2+8\sqrt{3}) \|f\|_{bmo^c(\P)} \pl .
 \end{align*}
Thus all the norms are equivalent on $\N\cup L_2(\N)$. \qd

\section{Bounded Fourier Multipliers on BMO}

In this section we prove the $BMO$ boundedness for  certain
singular integrals obtained as a function of the generator for
arbitrary semigroups. The ideas for the proof   can be traced back
to E. Stein's universal $L_p$-bounded for Fourier multipliers.
\begin{lemma}\label{convex gamma} Let $\Gamma$ be the gradient form associated with a standard semigroup $S_t$. Then
\begin{eqnarray}
\Gamma (\int_{\Omega} f_td\mu (t),\int_{\Omega} f_td\mu (t))\leq
\int_\Omega |d\mu (t)|\int_{\Omega} \Gamma (f_t,f_t)|d\mu (t)|, \label{convex gammain}
\end{eqnarray}
for $\N$-valued function $f$ on a measure space $\{\Omega, \mu\}$
such that $\Gamma(f_t,f_t)$ is weakly measurable.

Let
$P_t$ be the Poisson semigroup subordinated to a standard semigroup
$T_t$ satisfying $\Gamma^2\geq0$. Then,
\[
{\Gamma }(v^n\partial^n P_vf,v\partial^n P_vf)\leq c_nP_{\frac v2}{\Gamma
}(f,f),
\]Here $n\in \nz$ and $\partial_t^nP(t)$ is the $n$-th
derivative of $P(t)$ with respect to $t$.
\end{lemma}

\begin{proof}
The bilinearity of $\Gamma$ implies
\begin{equation*}
\Gamma(\frac{x+y}2,\frac{x+y}2)+\Gamma(\frac{x-y}2,\frac{x-y}2)=\frac12(\Gamma(x,x)+\Gamma(y,y)),
\end{equation*}
for $x,y\in \A$.
Note $\Gamma(x,x)\geq0$. We have
\begin{equation*}
\Gamma(\frac{x+y}2,\frac{x+y}2)\leq\frac12(\Gamma(x,x)+\Gamma(y,y)),
\end{equation*}
for $x,y\in \A$.
(\ref{convex gammain}) follows by the convexity of $\Gamma(\cdot,\cdot)$, which we just proved.

By (\ref{ps-form}), we may write $v^n\partial^n P_v$ as $\int_0^\infty T_{\tau} d\mu_n (\tau )$ with $%
\int_0^\infty |d\mu_n (\tau )|\leq c_n$ and $\int_0^\infty T_{\tau}
|d\mu_n (\tau )|\leq cP_{\frac v2}. $ We deduce  from (\ref{convex
gammain}) and $\Gamma^2\gl 0$ that
\begin{align*}
{\Gamma }(v^n\partial^n P_vf,v^n\partial^n P_vf) &\leq c_n\int_0^\infty
{\Gamma }(T_{\tau} f,T_{\tau} f)|d\mu_n (\tau )| \leq c_n\int_0^\infty
T_{\tau} {\Gamma }(f,f)|d\mu_n (\tau )| \leq
  c_nP_{\frac v2}{\Gamma }(f,f)\pl . \qedhere
\end{align*} \end{proof}

Recall that $\Gamma_{\partial_t^2}(f,g)=(\partial_t f^*)(\partial_t
g)$ is the gradient forms associated with $T_t=e^{t\partial_t^2}$
and satisfies $\Gamma^2_{\partial_t^2}\geq0$. According to  Lemma
\ref{convex gamma}, we know that
\begin{eqnarray}\label{convex partial}
|\partial_ v (v\partial_ vP_vf)|^2\leq cP_{\frac v2}|v{\partial_v
}P_vf|^2.
\end{eqnarray}
Since $\widehat{\Gamma}=\Gamma+\Gamma_{\partial_t^2}$, we obtain
 \begin{eqnarray}\label{convex gammahat}
\widehat{\Gamma}(vP_v'f,vP_v'f)\leq cP_{\frac
v2}\widehat{\Gamma}(P_vf,P_vf).
\end{eqnarray}

\noindent We now want to define singular integrals of the form
$F(A)$  where $F$ is a nice function. We follows Stein's idea and
assume that $F$ is given by a Laplace transform.  Let $a$ be a
scalar valued function such that
 \begin{equation}
 s\int_s^\infty\frac {|a(v-s)|^2}{v^2}dv\leq c_a^2 ,\label{ca}\end{equation}
 for all $s>0$ and some constant positive $c_a$. Define $M_a$ as
\[
M_a(f)=\int_0^\infty a(t)\frac {\partial P_tf}{\partial t}dt.
\]

\begin{lemma}\label{mulgamma} Assume $T_t$ be a standard semigroup
satisfying $\Gamma^2\geq0$. We have $$\|M_a(f)\|_{BMO^c(\Gamma)}\leq
cc_a\|f\|_{BMO^c(\Gamma)}.$$
\end{lemma}

\begin{proof} Let
$$S_t(f) =P_t\int_0^ts\Gamma (P_sf,P_sf)ds.$$
We simplify the notation by using ${\Gamma }[f]={\Gamma }(f,f)$. Let
us  compute $\|S_t(M_a(f))\|$:
\begin{align*}
\|S_t(M_a(f))\| =\|P_t\int_0^ts {\Gamma }[P_sM_a(f)]ds\|
 &=
 \|\int_0^tsP_t {\Gamma }[P_s\int_0^\infty a(v)\frac{\partial P_v}{%
\partial v}fdv]ds\| \\
&=\|\int_0^tsP_t {\Gamma }[\int_0^\infty a(v)\frac{\partial P_{v+s}}{%
\partial v}fdv]ds\| \\
 &=
  \|\int_0^tsP_t {\Gamma }[\int_s^\infty a(v-s)\frac 1vv\frac{%
 \partial P_v}{\partial v}fdv]ds\| \\
\text{(first inequality of Lemma  \ref{convex gamma})} &\leq \|\int_0^tP_t\bigg(s\int_s^\infty \frac{%
|a|^2}{v^2}dv\int_s^\infty  {\Gamma }[v\frac{\partial P_v}{\partial v}%
f]dv\bigg)ds\| \\
\text{(assumption (\ref{ca}))} &\le
  c_a^2\|\int_0^tP_t\bigg(\int_s^\infty  {\Gamma }[v\frac{\partial P_v}{%
\partial v}f]dv\bigg)ds\| \\
\text{(change of variables)} &=8c_a^2\|\int_0^tP_t\bigg(\int_{\frac
s2}^\infty
 {\Gamma }[v\frac{\partial P_v}{\partial v}P_vf]dv\bigg)ds\| \\
\text{(Lemma \ref{convex gamma})} &\leq
 cc_a^2\|\int_0^tP_t\bigg(\int_{\frac
s2}^\infty P_{\frac v2} {\Gamma }[P_vf]dv\bigg)ds\| \\
 &\le  cc_a^2\|\int_0^t\int_{\frac s2}^\infty P_{\frac v2+t} {\Gamma }[P_vf%
]dvds\| \\
({\rm Integrate}\ ds\ {\rm first}\ )&= cc_a^2\|\int_0^\infty P_{\frac v2+t} {\Gamma }[P_vf]\min (t,2v)dv\| \\
({\rm Lemma \ \ref{lem1}}\ ) &\leq  cc_a^2\|f\|^2_{BMO^c(\Gamma)}
\end{align*}
Taking the supremum over $t$, we obtain
\begin{align*}
\|M_a(f)\|_{BMO^c(\Gamma)} &=\sup_t\|S_t(M_a(f))\|^\frac12\leq
cc_a\|f\|_{BMO^c(\Gamma)} \qedhere .
\end{align*}
\end{proof}

Using (\ref{convex partial}), exactly the same proof shows that,
without assuming $\Gamma^2\geq0$,
 \begin{equation}\label{parest}
  \|M_a:BMO^c(\partial)\to BMO^c(\partial) \| \kl  cc_a.
  \end{equation}
The same technique also allows us to obtain
estimates for operators of the form
 \[ M_{a,n}=\int_0^\infty a(t)t^{n-1}\partial_t^nP(t)dt\pl .\]
 Let us state this
explicitly.

\begin{theorem}\label{mul1} Let $T_t$ be a standard semigroup. Then
 \begin{eqnarray} \|M_{a,n}(f)\|_{BMO^c(\partial)}\leq
c_nc_a\|f\|_{BMO^c(\partial)}.
\end{eqnarray}
If in addition, $T_t$ satisfies $\Gamma^2\geq0$, then
\begin{eqnarray} \|M_{a,n}(f)\|_{BMO^c(\Gamma)}\leq
c_nc_a\|f\|_{BMO^c(\Gamma)}.
\end{eqnarray}
\end{theorem}

\begin{cor}\label{mul} Let $T_t$ be a standard semigroup
satisfying $\Gamma^2\geq0$. Then
 \[ \|M_{a,n}(f)\|_{BMO^c(\P)}\leq c_nc_a\|f\|_{BMO^c(\P)}\pl .\]
\end{cor}
\begin{proof}
By Theorem \ref{last}, we know
\[
\|f\|_{BMO^c(\P)}\simeq \|f\|_{BMO^c(\hat{\Gamma})}.
\]
By the definition of $\hat{\Gamma}$, we see that
$$\|f\|_{BMO^c(\hat{\Gamma})}\simeq\max\{\|f\|_{BMO^c({\Gamma})},\|f\|_{BMO^c(\partial)}\}.$$
Therefore, Corollary \ref{mul} follows from Theorem \ref{mul1}.
\end{proof}

\begin{exam}\label{stein}Let $-\phi $ be a real valued, symmetric, conditionally
negative function on a discrete group $G$ satisfying $\phi (1)=0.$
Let $A$ be the unbounded operator defined on $\C[G]$ as
$$A(\lambda(g))=\phi(g)\lambda(g).$$ Let $T_t=\exp(-tA)$, i.e.
$$T_t(\lambda(g))=\exp(-t\phi(g))\lambda(g).$$
  $(T_t)_t$ extends to a standard semigroup of operators with generator $-A$ on the group von Neumann algebra ${\mathcal N}=VN(G)$ following Schoenberg's theorem. $T_t$  satisfies $\Gamma^2\geq 0$ too. Therefore,
  Theorem \ref{mul1} and Corollary \ref{mul} applies
in this setting. Here we note that $M_{a}$ is indeed a Fourier
multiplier. Indeed, assume that $m$ is a complex valued function of
the form
\[ m(g) \lel c\int_0^\infty \phi^{\frac 12}(g)e^{-t\phi ^{\frac
 12}(g)}a(t)dt \pl.
 \]
Then $M_{a}(\la(g))\lel m(g)\la(g)$. For example we may consider
$a(t)=t^{-2is}$ with $s$ a real number. Then we deduce that $
m(g)=\Gamma(1-is)[\phi (g)]^{is}$ is a Fourier multiplier. Note the
subordinated semigroup in this case is given by
 \[ P_t(\la(g))\lel e^{-t\sqrt{\psi}(g)}\la(g) \pl .\]
Therefore Corollary \ref{mul} imply that
 \begin{equation}\label {stein1}
 \|M_a(f)\|_{BMO^c(\P)}\leq cc_a\|f\|_{BMO^c(\P)}\pl .
 \end{equation}
for all $f\in L_2(VN(G))$.
\end{exam}



In the remaining  part of this article, we will use probabilistic
methods to prove an interpolation theorem for semigroup  BMO spaces.
This in turn allows us to obtain $L_p$ bounds for Fourier
multipliers of the form above.

\section{Probabilistic models for semigroup of operators}
In the section, we introduce BMO spaces for noncommutative martingales and P. A. Meyer's probabilistic model for semigroup of operators. We will apply them in the next section to an interpolation theorem for BMO associated with semigroups.

\subsection{Noncommutative martingales} ${\atop}$

Let $(\M,\tau)$ be a semifinite von Neumann algebra equipped with a
semifinite normal faithful trace $\tau$. We will say that an
increasing family $(\M_t)_{t\gl 0}$ is an \emph{increasing
filtration} if if $s<t$ implies $\M_s\subset \M_t$, $\bigcup_t \M_t$
is weakly dense, and the restriction of the trace is semifinite and
faithful for every $\M_t$.  We refer to \cite{Tak} for the fact that
this implies the existence of a uniquely determined trace preserving
conditional expectations $E_t:\M\to \M_t$. By uniqueness we have
$E_{s}E_t=E_{\min(s,t)}$. Right continuity, i.e. $\bigcap_{s>t}
\M_s=\M_t$ for all $t\gl 0$, will be part of the assumption when we
talk about increasing filtrations. Similarly, we will say that
$(\M_t)_{t\gl 0}$ is a \emph{decreasing filtration} if $s<t$ implies
$\M_s\supset \M_t$, $\M$ is the weak closure of $\bigcup_t \M_t$,
and we have left continuity. Again we have a family of conditional
expectations $E_s:\M\to \M_s$ such that $E_{s}E_t=E_{\max(s,t)}$. We
have $\M_0={\M}, E_0=id$ for decreasing filtration and set
$\M_\infty =\wedge_t\M_t$ as a convention. Set $\M_\infty=\M,
E_\infty=id,\M_0=\wedge_t\M_t$ for increasing filtration. A
(reversed) martingale adapted to $({\M}_t)_{t\in[0,\infty)}$ is a
family $(x_t)\in L^1(\M)+L^\infty(\M)$ such that $E_t(x_s-x_t)=0$
for any $s>t\geq0$ for increasing filtration ( for $t>s\geq0$ for
decreasing filtration).

For $x\in L_p(\M), 1\leq p\leq \infty$, the family $(x_t)$ given by
$x_t=E_tx$ is a martingale with respect to $\M_t$. For $2<p\leq
\infty$, we define
\smallskip\begin{center}
$\|x\|_{L^c_pmo(\M)} \lel
 \|\sup^+_t E_{t}(|x-E_tx|^2)\|_{{\frac p2}}^{\frac12},
  $\end{center} \smallskip
\noindent where $\|\sup^+\cdot\|_{\frac p2}$ should be understood in
the sense of vector-valued noncommutative $L_p$ spaces, see
\cite{pvp}, \cite{JD}, \cite{JX3}). Let
$$\|x\|_{L_pmo(\M)}=\max\{\|x\|_{L^c_pmo(\M)},\|x^*\|_{L^c_pmo(\M)}\}.$$
By Doob's inequality, we know  that
\begin{eqnarray}\label{Doob}
\|x\|_{L_pmo(\M)}\leq c_p\|x\|_{L_p(\M)}.
\end{eqnarray}
Let $L_p^0({\mathcal M}), 1\leq p\leq\infty$, be the quotient space
of $L_p(\M)$ by $\{x,x=\E x\}$. Here $\E$ is the projection from
$\M$ onto $\wedge_t\M_t$, which equals to $E_0$ in the case of
increasing filtration and equals to $E_\infty$ in the case of
decreasing filtration. For $2<p<\infty$, let $L_p^0mo(\M)$
($L_p^{c,0}mo(\M)$) be the completion of $\M^0=L^0_\infty(\M)$ by
$\|\cdot\|_{L_pmo(\M)}$ ($\|\cdot\|_{L_p^cmo(\M)}$)-norm. For
$p=\infty$, we have to consider a weak$^*$ completion and denote the
completed spaces by $bmo^{c}(\M)$ (resp. $bmo(\M)$). We refer the
interested readers to \cite{JKPX} and \cite{JuPe} for more
information on noncommutative martingales with continuous
filtrations.

We now introduce martingale $h_q$-space, which are preduals of
$L_pmo'$s. Let $\si=\{0=s_0<s_1,...s_{n-1}<s_n=\infty\}$ be a finite
partition of $[0,\infty]$. For $x\in L^1(\M)+L^\infty(\M)$, define
the conditioned bracket $\langle x,x\rangle(\si)$ ($k\leq n$) as
  $$\langle x,x\rangle(\si)=\sum_{j=1}^n E_{s_{j-1}}|E_{s_{j}}x-E_{s_{j-1}}x|^2.$$
   The $h_p^{c}(\si), 1\leq p <\infty, $-norm of $x$ is defined as
$$\|x\|_{h_p^c}=\|(\langle x,x\rangle(\si))^{\frac12}\|_{L_p}.$$
Let $\U$ be an ultrafilter refining the natural order given by
inclusion on the set of all partitions of $[0,\infty]$. The
$h_p^{c}(\U)$ and $h_p^r(\U)$-norms of $x$ are defined as
\begin{eqnarray*}
\|x\|_{h_p^c}&=&\lim_{\si,\U}\|(\langle
x,x\rangle(\si))^{\frac12}\|_{L_p}, \quad
\|x\|_{h_p^r} \pl =\pl \|x^*\|_{h_p^c}.
\end{eqnarray*}
 The $h_p^d(\U)$-norm of $x$ is
defined as
$$\|x\|_{h_p^d}=\lim_{\si,\U}(\sum_{s_j\in \si}\|E_{s_j}x-E_{s_{j-1}}x\|^p_{L_p})^\frac1p.$$
 It is proved in \cite{JuPe}
that these norms do not depend on the choice of $\U$ whenever $\U$
is containing the filter base of tails. Let $h_{p}^{c}(\M)$
($h_{p}^{r}(\M)$, $h_{p}^{d}(\M)$ ) be the collection of all $x$
with finite $h_p^{c}(\U)$ ($h_{p}^{r}(\M)$, $h_{p}^{d}(\U)$ )-norm.
It is proved in \cite{JuPe} that
\begin{eqnarray*}
(h_{p}^{c}(\M))^*=L_q^cmo(\M)=h_q^c(\M), 1\leq p<2, \frac1p+\frac1q=1\\
h_{p}^c(\M)+h_p^r(\M)+h_p^d(\M)=L_p(\M), 1<p<2.
\end{eqnarray*}
Denote by $h_p(\M)=h_p^c(\M)+h^r_p(\M)$, $H_p(\M)=h_p(\M)+h_p^d(\M),
1\leq p<2$, $h_p(\M)=h_p^c(\M)\cap h^r_p(\M)$ for $2\leq p<\infty,$
and $BMO(\M)=(H_1(\M))^*=bmo(\M)\cap (h_1^d(\M))^*$, we have
$$[BMO(\M),L_1(\M)]_{\frac1q}=L_q(\M),$$ for all $1<q<\infty$.

Recall that a martingale $x=(x_t)_t$ has \emph{a.u.\!\! continuous }
path provided that, for every $T>0, \eps>0$ there exists a
projection $e$ with $\tau(1-e)<\eps$ such that the function
$f_e:[0,T]\to {\mathcal M}$ given by
 \[ f_e(t) \lel x_te \in {\mathcal M} \]
is continuous. The following observation will be crucial for us.

\begin{lemma}\label{ctmar} Let $x^{\la}$ be a net
of martingales in $\M\cap L_2(\M)$ with a. u. \! continuous path.
Suppose $x^{\la}$ weakly converges in $L_2(\M)$ and the limit $x$ is
in $bmo$. Then $x\in BMO$ and
 \[ \|x\|_{BMO}\kl c\|x\|_{bmo} \pl .\]
Moreover, let $p>2$ and $x\in L_p(\M)$ with a.u. \!continuous path. Then $\|x\|_{h_p}\simeq_{c_p} \|x\|_{L_p(\M)}$.
\end{lemma}

\newcommand{\au}{a.\! u.\! }

\begin{proof} We first prove that for martingales $x\in L_2(\M)\cap\M$ with a. u.\! continuous
path, we have
\begin{eqnarray}\|x\|_{(h_1^d(\M))^*}=0.\label{h1d*}\end{eqnarray} By Doob's inequality for noncommutative martingales,
one can show that  \au continuity and $x\in L_2(\M)\cap \M$ imply
that $x_t=f(t)a$ for some $a\in L_q(\M)$ and a continuous function
$f:[0,T]\to \M$ for any $q>2, T<\infty$. This implies that
 \[ \lim_{\si,\U}
 \|d_{t_j}(x)\|_{L_q(\ell_\infty^c)} = \lim_{\si,\U}
 \|{\sup}^+_{t_j\in \si}d_{t_j}(x)^*d_{t_j}(x)\|_{q/2} \lel 0 \pl
 ,\] for any ultrafilter $\U$ of $[0,T]$ containing the filter base of tails.
 Note that
 \begin{eqnarray}\label{hqd}  \|d_{t_j}((x)^*)\|_{L_q(\ell_\infty^c)}
 \kl   \|d_{t_j}(x)\|_{L_q(\ell_q)}
 \kl \|d_{t_j}(x)\|_{L_q(\ell_{\infty}^c)}^{1-\theta}
  \|x\|_{H_q^c(\si)}^{\theta} \end{eqnarray}
for $\theta=\frac2q$. Thus we also find that
 \[ \lim_{\si,\U}
 \|d_{t_j}(x)\|_{L_q(\ell_\infty^r)} \lel 0 \pl .\]
We recall from \cite{JuPe}  that $\bigcup_{p>1}
B_{h_{p}^{1_c}+h_{p}^{1_r}}\subset h_1^d$ are dense in the unit ball
of $h_1^d$.  Here the $h_{p}^{1_c}$ is defined such that the norm of
$x\in L_2(\M)\cap(h_{p}^{1_c})^*$ is given by $\lim_{\si,\U}
 \|d_{t_j}(x)\|_{L_q(\ell_\infty^c)}.$
Therefore, $x$ satisfies (\ref{h1d*}) if $x$ is in $\M\cap L_2(\M)$ and has \au continuous path. Now,  let $x^{\la}$ be a net
of weakly $L_2$-converging martingales in $\M\cap L_2(\M)$ with a.
u. \! continuous path. Suppose its weak $L_2$-limit $x$ is in $bmo$.
Recall from \cite{JuPe} that, for any $y\in H_1^c\cap L_2(\M) $ we
may find a decomposition such that $y=y_1+y_2$ with $y_1\in
h_1^c\cap L_2(\M), y_2\in h_1^d\cap L_2(\M)$ and
$\|y_1\|_{h_1^c}+\|y_2\|_{h_1^d}\leq 2\|y\|_{H_1^c}$. Then
 \begin{align*}
 |\tau(y^*x)|&\le  |\tau(y_1^*x)|+
 |\tau(y_2^*x)|
  \lel |\tau(y_1^*x)|+|\lim_{\la} \tau(y_2^*x^{\la})| \lel |\tau(y_1^*x)|\leq c\|x\|_{bmo}\|y_1\|_{h_1^c} \pl .
  \end{align*}
  Since  the unit ball of $H_1^c(\M)\cap L_2(\M)$ in dense in the unit
  ball of $H_1^c(\M)$, we get
  $$\|x\|_{bmo_c}\leq c\|x\|_{BMO_c}.$$
From (\ref{hqd}) we have already seen  that for martingales $x\in
L_q(\M)$ with continuous path we have $\lim_{\si,\U}
\|x\|_{h_q^d(\si)}=0$ because
$\|x\|_{h_q^d(\si)}=\|d_{t_j}(x)\|_{L_q(\ell_q)}$. Hence  we have
 \[ \|x\|_{H_q^c} \kl C \|x\|_{h_q^c}  \]
for $q>2$ because $H_q^c=h_q^c\cap h_q^d$ for $q>2$.  \qd

In the previous argument we learned
for continuous martingales with a.u.\!\! continuous path we have $\|x\|_{h_p^d}=0$ (see also \cite{JKPX}). In fact, in this paper we might simply take this as a definition. We will need some more results in this direction and state them in the following lemma.

\begin{lemma}\label{JuPe} Let $1<p<\infty$. We have
\begin{eqnarray}
(BMO(\M),L_1(\M))_{\frac1p}=L_p(\M),\\
(bmo^c(\M),L_2(\M))_{\frac2p}=h_p^c(\M)
\end{eqnarray}
with equivalence constants $\simeq p$. Suppose that $x\in L_p(\M),
2<p<\infty$ and $(E_tx)_t$ is a.u. continuous. We have
\begin{eqnarray}
\|x\|_{L_pmo(\M)}+\|\E x\|_{L_p(\M)}\simeq\|x\|_{L_p(\M)},\label{eq}
\end{eqnarray} with
equivalence constants $\simeq p$ for $p>4$.
\end{lemma}

We say that a standard semigroup $(T_t)$ on a semifinite von Neumann
algebra $\N$ admits  {\it a standard Markov dilation} if there
exists a larger semifinite von Neumann algebra $\M$, an increasing
filtration $(\M_{s]})_{s\gl 0}$ and  trace preserving
$^*$-homomorphism $\pi_s$ such that
  \[ E_{s]}(\pi_t(x))\lel \pi_s(T_{t-s}x)  \quad s<t\pl ,\pl
  x\in {\mathcal \N} \pl .\]
We say that $(T_t)$ admits  {\it a reversed Markov dilation} if
there exists a larger von Neumann algebra $\M$, a decreasing
filtration $(\M_{[s})_{s\gl 0}$, and trace preserving
$^*$-homomorphisms $\pi_s:\N\to \M_{[s}$ such that
  \[ E_{[s}(\pi_t(x))\lel \pi_s(T_{s-t}x)  \quad t<s\pl, \pl x\in {\mathcal
  N}\pl . \]
We say that $(T_t)$ admits  a {\it Markov dilation} if it admits
either a standard dilation or a reversed dilation. We refer to
\cite{Ku} for related questions. A glance at (\ref{ps-form}) shows
that a Markov dilation for $(T_t)$ implies that the $P_t$'s are
factorable (in the sense of \cite{claire}). According to
\cite{claire}, we know that a Markov dilation for $T_t$ (standard or
reversed) yields a Markov dilation (standard and reversed) for
$P_t$.

In the noncommutative setting the existence of a Markov dilation is
no longer for free, as it is in the commutative case. We refer the
reader to \cite{Ri} for its existence for group von Neumann algebra
and to \cite{JRS} for its existence for finite von Neumann algebra.
However, the existence of a Markov dilation allows us to use
probabilistic tools for semigroups of operators. In particular,
given a a reversed Markov dilation we know that
$m(x)=(m_s(x))_{s\geq0}$ with
\begin{eqnarray}
 m_s(x)=\pi_s(T_sx), \label{ms} \end{eqnarray}
is a martingale with respect to the reversed filtration $(\M_{[s})$.
A standard Markov dilation implies that, for any $v>0$,
$m(x)=(m_s(x))_{v\geq s\geq0}$ with
\begin{eqnarray}
m_s(x)=\pi_s(T_{v-s}x)\end{eqnarray} is a martingale with respect to
the standard filtration  $(\M_{s]})$.

\begin{prop}\label{bmo1}  Let $(T_t)$ be a standard semigroup of operators on $\N$ with reversed
Markov dilation $(\pi_t,\M_t)$. Let $x\in L_p(\N)$. Then
$E_{[s}(\pi_0x)=\pi_sT_s x$ and
\begin{eqnarray}
\|\pi_0(x)\|_{L_p^cmo(\M)}=\|{\sup}_t^+\pi_t(T_t|x|^2-|T_tx|^2)\|_{\frac
p2}^{\frac12}, \label{mxbmo}
\end{eqnarray}
for $2<p\leq\infty$.
In particular,
\begin{eqnarray}
\|\pi_0(x)\|_{bmo^c(\M)}=\|x\|_{bmo^c(\T)}.\label{rmartingale}
\end{eqnarray}
\end{prop}

\begin{proof} 
To prove (\ref{mxbmo}), we apply the reversed dilation condition and get $E_{[t}\pi_0(x)=\pi_tT_tx$. Then
\begin{eqnarray*}
   E_{[t}|\pi_0(x)|^2-|E_{[t}\pi_0(x)|^2=\pi_t(T_t|x|^2-|T_tx|^2).
   \end{eqnarray*}
Taking the supremum over  all $t$, we obtain  the assertion. The equation
(\ref{rmartingale}) now follows immediately from the  definition.\qd

\subsection{Meyer's probabilistic model for semigroup of operators}
Meyer's probabilistic model provides another way to connect
semigroups of operators with martingales. Let us start with an
observation due to Meyer \cite{Me1}.

\begin{prop}\label{meyer1} Let
$T_t$ be a semigroup with a standard (resp. reversed) Markov
dilation $(\pi_t,\M_t)$. For $x\in \dom(A)$, let
$n(x)=(n_s(x))_{s\geq0}$ with
\begin{eqnarray}
n_s(x)=\pi_s(x)+\int_0^s\pi_v(A(x))dv, \label{ns}\end{eqnarray} for
standard Markov dilation and
\begin{eqnarray}
n_s(x)=\pi_s(x)+\int_s^\infty\pi_v(A(x))dv, \end{eqnarray} for
reversed Markov dilation. Then $n(x)$ is a (reversed) martingale
with respect to the filtration $\M_s$.
\end{prop}

\begin{proof} Let us verify that $E_{t}n_s(x)=n_t(x)$ for $t>s$ in
the reversed dilation case. The verification for $t<s$ in the
standard dilation case is similar. Due to the dilation property we
have
\begin{align*}
 &E_t\bigg(\pi_s(x)+\int_s^\infty\pi_v(A(x)) dv\bigg)\lel \pi_t(T_{t-s}(x))+\int_s^t
\pi_t(T_{t-v}A(x))dv+ \int_t^\infty\pi_v(A(x))dv\\
 &= \pi_t(T_{t-s}(x))+\int_s^t
\pi_t(\partial_vT_{t-v}(x))dv+ \int_t^\infty\pi_v(A(x))dv
 \lel \pi_t(x)+
\int_t^\infty\pi_v(A(x))dv.
\end{align*}
This means $n(x)$ satisfies the martingale property
$E_tn_s(x)=n_t(x)$ for $t>s$.
\end{proof}

The main ingredient in Meyer's model  is to use L\'{e}vy's stopping
time argument for the Brownian motion (see however \cite{GuCR},
\cite{GuVa} for more compact notation).
 Given a standard semigroup $T_t$ with generator $A$, assume $(T_t)$ admits a standard Markov dilation $(\pi_s,\M_s)$.
 We consider a new generator
 \[ \hat{A}\lel -\frac{d^2}{dt^2}+ A \pl .\]
 defined densely on
$$L^2({\Bbb R})\ten L^2(\N).$$
This leads to a new semigroup of operators $\hat{T}_t=\exp(-t\hat
{A})$ such that
 \begin{eqnarray*}
 \hat{\Gamma}(f(t),g(t))\lel \Gamma(f(t),g(t))+\frac{df^*(t)}{dt}\frac{dg(t)}{dt}.\end{eqnarray*}
Let $(B_t)$ be a classical one dimensional Brownian motion
associated with $dt$ (instead of the usual $\frac12 dt$ in the
stochastic differential equation) such that $B_0=a$ holds with
probability $1$. Let $\hat{M}_s\lel M_s^B\ten \M_s$ with $M_s^B$ the
von Neumann algebra of the Brownian motion observed until time $s$.
The Markov dilation for the new semigroup $\hat{T}_t$ is given by
$\hat{M}_t$ and $\hat{\pi}_t(f(\cdot))\lel \pi_tf(B_t(\cdot))$.

For $x\in L_p(\N), 1\leq p\leq\infty$, denote by $Px$  for the $L_p({\mathcal
N})$-valued function on $[0,\infty)$
 \[ Px(t)\lel P_t(x) \pl .\]
Recall that we write $P'x$ for the functions
$\frac{d}{dt}P_tx$. Fix a real number $a>0$. Let ${\bf t}_a$ be the stopping time of $B_t$ first hit $0$, i.e.
$${\bf t}_a=\inf\{t:
B_t(\om)=0\}.$$
The following observation due to P. A. Meyer.
\begin{prop}\label{nx} For any $x\in L_p(\N), 1\leq p\leq\infty,$
$$\hat{n}_a(x)=(\hat{\pi}_{\ttt_a\wedge t} Px)_t$$ is a martingale with respect to the filtration
$$\hat{\mathcal M}_{t,a}=\bigvee_{v\leq t}\hat{\pi}_{\ttt_a\wedge v}({\mathcal N}\otimes
L^\infty(\rz)).$$
\end{prop}
\begin{proof}
Apply Proposition \ref{meyer1} to $\hat{A}, \hat{\pi_t}$, we get
$\hat{\pi}_{t} (Px)$ is a martingale because $\hat{A}(Px)=0$.
Therefore, $\hat{\pi}_{\ttt_a\wedge t} (Px)$ is a martingale too
since $\ttt_a$ is a stopping time.
\end{proof}
Let
$$\hat{\M}_a=\overline{\vee_{t\geq0}
\hat{\M}_{a,t}}^{wot}=\overline{{\vee_{\ttt_a\geq t\geq0}
\M_t}}^{wot}.$$ Let $\hat{E}_{t}$ be the conditional expectation
from $\hat{\M}_a$ onto $\hat{\M}_{a,t}$. Proposition \ref{nx}
implies that
\begin{eqnarray}
\hat{E}_{t}(\pi_{\ttt_a}x)=\hat{\pi}_{t_a\wedge t}
Px=\pi_{\ttt_a\wedge t} P_{B_{\ttt_a\wedge t}}x,
\end{eqnarray}
for any $x\in \N$.

Meyer's model allows us to consider martingale spaces with respect
to the time and space filtrations simultaneously. Let
$L_p^{c,0}mo(\hat{\mathcal M}_a), L_p^0mo(\hat{\mathcal M}_a),
2<p\leq\infty$ be the martingale spaces with respect to the
filtration $\hat{\M}_{a,t}$. Recall that we have an orthogonal
projection $P_{br}$ on the subspace consisting of martingales
$x=(x_t)$ with the form \begin{eqnarray}\label{br}
x_t=\int_0^{t\wedge \ttt_a} y_sdB_s \end{eqnarray} with $y_s$
adapted to $\M_{s}$. By approximation, we see that $(x_t)_t$ has
continuous path, i.e. $x_t$ is continuous on $t$ with respect to the
$L_p$-norm, provided $\sup_t\|y_t\|_{L_p(\hat\M_a)}<\infty.$

Denote $P_{\Gamma}=I-P_{br}$. Recall that it is our convention to
write $bmo$ instead of $L_\infty mo$ and $BMO=bmo\cap (h_1^d)^*$ .

\begin{lemma}\label{bmo2} Let $(T_t)$ be a standard semigroup admitting a Markov dilation, $(P_t)_t$ the semigroup subordinated to $(T_t)_t$ and $f\in
\mathcal{N}\cup L_2(\N)$. Then
\begin{enumerate}
 \item[(i)] $\|f\|_{bmo^c(\P)} \lel \|\hat{n}_a(f)\|_{bmo^c(\hat{\M_a})}$.
 \item[(ii)]$
 \|P_{br}\hat{n}_a(f)\|_{BMO^c(\hat{\M_a})}\simeq  \|P_{br}\hat{n}_a(f)\|_{bmo^c(\hat{\M_a})} \simeq \sup_b
   \|\int_0^{\infty}\!\!
  P_{b+s}|\frac{\partial P_{s}f}{\partial s}|^2 \min(s,b) ds\|^{\frac
  12}$
  .
 \item[(iii)] If in addition  $\Gamma^2\gl 0$, then
  \begin{align*}
   \|P_{\Gamma}\hat{n}_a(f)\|_{bmo^c(\hat{\M_a})} &\simeq \sup_b
   \|\int_0^{\infty}\!\!
  P_{b+s}\Gamma(P_{s}f,P_{s}f) \min(s,b) ds\|^{\frac 12 }. \end{align*}
\end{enumerate}
\end{lemma}

\begin{proof} We recall that $(\hat{n}_a(f))_t=
\hat{\pi}_{\ttt_a\wedge t}(Pf)=\pi_{\ttt_a\wedge
t}(P_{B_{\ttt_a\wedge t}}(f))$. So the end element
$\hat{n}_a(f)=\pi_{\ttt_a}(f)$,
$(\hat{n}_a(f))_0=\pi_{0}(P_{B_0}(f))=\1(\omega)\otimes\pi_0P_af$.
Hence we get
 \begin{align*}
  \hat{E}_t(|\hat{n}_a(f)|^2)-|\hat{E}_t(\hat{n}_a(f))|^2
  &= \pi_{\ttt_a\wedge t}(P_{B_{\ttt_a\wedge t}}|f|^2-|P_{B_{\ttt_a\wedge
  t}}f|^2) \pl,
  \end{align*}
for $\ttt_a(\om)>t$. Thus in any case we have
 \[ {\rm ess}\sup_{\om} \|\hat{E}_t(|\hat{n}_a(f)|^2)-|\hat{E}_t(\hat{n}_a(f))|^2\| \kl \sup_s \|\pi_{\ttt_a\wedge
 t}(P_s|f|^2-|P_sf|^2)\|\kl \|f\|_{bmo^c(\P)}^2 \pl .\]
However, recall that $B_0(\om)=a$ almost everywhere. This means
$B_t=a+\tilde{B}_t$ where $\tilde{B}_t$ is a centered Brownian
motion. Since $\limsup_t |\tilde{B}_t|/\sqrt{2t\log\log t}=1$, we
know that with probability $1$ the process $|\tilde{B}_t|$ exceeds
$a$. Thus with probability $1$ the process $B_t$ hits $0$ or $2a$.
Hence with probability $\frac12$ the process hits $2a$ before it
hits $0$. Let us assume that $B_{t(\om)}(\om)=2a$ and $B_s(\om)>0$
for $0<s<t(\om)$. By starting a new Brownian motion at $t(\om)$, we
see with conditional probability $\frac12$ we have $B_{t'(\om)}=4a$
for some $t(\om)<t'(\om)$ and $B_s(\om)>0$ for all
$t(\om)<s<t'(\om)$. By induction we deduce  that with probability
$2^{-n}$ the process $B_t$ hits $2^na$ before it hits $0$. Thus
given any $b>0$, we may choose $n$ such that $2^{n}a>b$. We see that
with positive probability there exists $t_n(\om)$ such that
$B_{t_n(\om)}=2^na$ and $B_s(\om)>0$ on $[0,t_n(\om)]$ and $B_s$ is
continuous. By continuity there exists $t(\om)\in
[t_n(\om),\ttt_a(\om)]$ such that $B_{t_{\om}}=b$. In particular,
 \[ \|\hat{E}_{t(\om)}(|\hat{n}_a(f)|^2)-|\hat{E}_{t(\om)}\hat{n}_a(f)|^2\| \lel \|\pi_{t(\om)}(P_{B_{t(\om)}}|f|^2-|P_{B_{t(\om)}}f|^2)\|
 \lel \|P_b|f|^2-|P_bf|^2\| \pl . \]
Taking the supremum over all $b$ yields (i). For the proof of (iii)
we first apply Lemma 2.5.5 and Lemma 2.5.10 (ii)  of \cite{JMr}(note
there, $\rho_a$ denotes for $\hat{n}_a$) . This immediately yields
the first inequality (after a concise review of the involved
constant for $\beta=\frac{2}{3}$). For the upper estimate of this
term, we recall that with positive probability every value $b$ is
hit. Then we start in equality (3.20) for a fixed $b=B_t(\om)$. We
use the monotonicity $\frac{P_{b+s}(z)}{b+s}\kl \frac{P_t(z)}{t}$
and find
 \begin{align*}
 &\ez\int_0^{\ttt_b} T_s(\Gamma(P_{\tilde{B}_s}x,P_{\tilde{B}_s}x)) ds
 \lel \frac{1}{2} \int_0^{\infty} \int_{|b-s|}^{b+s} P_{t}\Gamma(P_sx,P_sx)
 dt ds \\
 &\gl   \frac{1}{2} \int_0^{\infty}
\frac{P_{b+s}\Gamma(P_sx,P_sx)}{b+s}
 \kla \int_{|b-s|}^{b+s} t dt \mer ds \lel
  \int_0^{\infty} \frac{P_{b+s}\Gamma(P_sx,P_sx)}{b+s} bs \pl
 ds \\
 &\gl \frac{1}{2} \int_0^{\infty} P_{b+s}\Gamma(P_{s}x,P_{s}x) \min(b,s)
 ds \pl .
 \end{align*}
The proof of the second equivalence of (ii) uses Lemma 2.5.10 (i) of
\cite{JMr} and is similar to (iii) but we only need $|P_tz|^2\le
P_t|z|^2$ instead of $\Gamma^2\gl 0$. The first equivalence of (ii)
follows from Lemma \ref{ctmar} and the fact that $P_br\hat{n}_a(f)$
has continuous path.
\end{proof}

\begin{lemma}\label{interbr} For any $y\in
L_p(\N), 2<p<\infty$, we have
\begin{eqnarray}\label{eqbr}
\|P_{br}\pi_{\tau_a}y\|_{L_{p}mo(\hat{\M}_a)}+\|
P_ay\|_{L_p(\hat{\M}_a)}\simeq\|y\|_{L_p(\N)},
\end{eqnarray}
with equivalent constants $\simeq p$ for $p>4$. Assume that $(\hat
E_{t}\pi_{\tau_a}y)_t$ is a.u. continuous, then
\begin{eqnarray}\label{eqbr1}
\|P_\Gamma \pi_{\tau_a}y\|_{L_pmo(\hat{\M}_a)}+\|
P_ay\|_{L_p(\hat{\M}_a)} \simeq\|y\|_{L_p(\N)}
\end{eqnarray}
with equivalence constants $\simeq p$ for $p>4$.
\end{lemma}

\begin{proof} This follows from the fact that $P_{br}\hat E_t\pi_{\ttt_a}y$ has continuous path, $\hat E_0\pi_{\tau_a}y=\pi_0P_a$,
Lemma \ref{JuPe} of this article, and Lemma 2.5.11 of \cite{JMr}
(Note $\pi_{\tau_a}$ is denoted by $\rho_a$ there).
\end{proof}

\subsection{Noncommutative Riesz Transforms}
We will prove a $L^\infty$-BMO boundedness for the noncommutative
Riesz transforms studied in \cite{JMr} in the first subsection.

Recall that the classical Riesz transforms on ${\Bbb R}^n$ can be
viewed as $\partial\cdot (-\triangle)^{-\frac12}$. Given a standard
semigroup of operators $T_t=e^{-tA}$, it is P. A. Meyer's idea to
view the generator $A$ as an analogue of $-\triangle$ and the
associated bilinear form $\Gamma(f,f)$ as $|\partial f|^2$. The
generalized Riesz transform of a function $f$ is
$[\Gamma(A^{\frac12}f,A^{\frac12}f)]^{\frac12}$. As a noncommutative
extension of Meyer's result, Junge/Mei proved in \cite{JMr} that
$$\|A^{\frac12}f\|_{L_p(\N)}\leq c_p\|[\Gamma(f,f)]^\frac12\|_{L_p(\N)},$$
for $2<p<\infty$ and self adjoint elements $f\in L_p(\N)$ with
additional assumptions on $T_t$. We will extend this
$L_p$-boundedness to $L^\infty-BMO$ boundedness.

\begin{theorem}\label{meyer} Assume $T_t$ admits a Markov dilation and satisfies $\Gamma^2\geq0$, we have
$$\|A^{\frac12}g\|_{BMO^c({\Gamma})}\leq c\max\{\|[\Gamma(g,g)]^\frac12\|,\|[\Gamma(g^*,g^*)]^\frac12\|\},$$
for $g\in \A$.
\end{theorem}
\begin{proof}
By the assumption of a Markov dilation $(\pi _t,E_t)$ of a standard
semigroup $T_t=e^{-tA},$ we have
\[
E_t\pi _uf-\pi _tf=\pi _t(T_{u-t}f-f)=\pi _t\int_t^u\frac{\partial T_{r-t}}{%
\partial r}fdr=-E_t\int_t^u\pi _rAfdr,
\] for $f\in \dom(A)$.
Apply to the Markov dilation $(\widehat{\pi }_t,\widehat{E}_t)$ of
the new semigroup $\widehat{T}_t=e^{-t\widehat{A}},$ we get
\[
\widehat{E}_t\widehat{\pi }_uf-\widehat{\pi }_tf=-\widehat{E}_t\int_t^u%
\widehat{\pi }_r\widehat{A}fdr.
\]
Passing to the stopping time $\ttt_a,$ we get
\[
\widehat{E}_t\widehat{\pi }_{\ttt_a}f-\widehat{\pi }_{\ttt_a\wedge t}f=-%
\widehat{E}_t\int_{\ttt_a\wedge t}^{\ttt_a}\widehat{\pi
}_r\widehat{A}fdr.
\]
For a given self adjoint $g\in \A$, let
\[
f(s)=\Gamma (P_sg,P_sg),
\]
It is an easy calculation by definition of $\Gamma ^2$ that
\[
-\widehat{A}f=2\Gamma ^2(P_sg,P_sg)+2\Gamma (P_s^{\prime
}g,P_s^{\prime }g).
\]
By $\Gamma ^2\geq 0,$ we have
\[
-\widehat{A}f\geq 2\Gamma (P_s^{\prime }g,P_s^{\prime }g)\geq 0.
\]
By Lemma 2.5.10 (ii) of \cite{JMr} (note $\rho_a$ denotes the same
martingale of $\hat{n}_a$),
\begin{eqnarray*}
\left\| P_\Gamma\hat{n}_a(g)\right\| _{bmo^c} 
=\sup_t\|\widehat{E}_t\int_{\ttt_a\wedge t}^{\ttt_a}\widehat{\pi }%
_r\Gamma (P_sg,P_sg)dr\|.
\end{eqnarray*}
Therefore, by Proposition \ref{bmo2},
\begin{align*}
\left\| A^{\frac 12}g\right\|^2 _{BMO^c(\Gamma )} &\approx \left\|
P_\Gamma \hat{n}_a(A^{\frac 12}g)\right\|^2 _{bmo^c} \lel
 \sup_t\|\widehat{E}_t\int_{\ttt_a\wedge t}^{\ttt_a}\widehat{\pi }%
_r\Gamma (P_s^{\prime }g,P_s^{\prime }g)dr\| \\
&\leq \sup_t\|-\widehat{E}_t\int_{\ttt_a\wedge t}^{\ttt_a}\widehat{\pi }_r%
\widehat{A}fdr\|
\lel \sup_t\|\widehat{E}_t\widehat{\pi }_{\ttt_a}f-\widehat{\pi
}_{\ttt_a\wedge t}f\|
\\
&\leq \|\widehat{\pi }_{\ttt_a}f\|+\sup_t\|\widehat{\pi }_{\ttt_a\wedge t}f\|
\lel \|\widehat{\pi }_{\ttt_a}\Gamma (Pg,Pg)\|+\sup_t\|\widehat{\pi
}_{\ttt_a\wedge
t}\Gamma (Pg,Pg)\| \\
&\leq \|\widehat{\pi }_{\ttt_a}P\Gamma (g,g)\|+\sup_t\|\widehat{\pi }%
_{\ttt_a\wedge t}P\Gamma (g,g)\| \\
&=\|\widehat{\pi }_{\ttt_a}P\Gamma (g,g)\|+\sup_t\|\widehat{E}_t\widehat{\pi }%
_{\ttt_a}P\Gamma (g,g)\| \\
&\leq 2\|\Gamma (g,g)\|.
\end{align*}
For non-self adjoint $g$, we obtain the desired inequality by
splitting $g=\frac {g^*+g}2 +i\frac {ig^*-ig}2$.
\end{proof}
\begin{cor}
Let $T_t$ be a standard semigroup satisfying $\Gamma^2\geq0$ and
admitting an a.u. continuous Markov dilation (see definition at the
beginning of the next section). We have
\begin{eqnarray*}
  \|A^\frac12g\|_{L_p(\N)}\leq
  cp\max\{\|\Gamma(g,g)\|_{L_p(\N)},\|\Gamma(g^*,g^*)\|_{L_p(\N)}\},
 \end{eqnarray*} for $2<p<\infty$.
\end{cor}
\begin{proof}
By the same argument used in the proof of Theorem \ref{meyer}, we
have
\begin{eqnarray*}
  \|P_\Gamma \hat n_a(A^\frac12g)\|_{L_{p}mo(\hat\M_a)}
  \leq c\max\{\|\Gamma(g,g)\|_{L_p(\N)},\|\Gamma(g^*,g^*)\|_{L_p(\N)}\}.
 \end{eqnarray*}
 We then obtain
\begin{eqnarray*}
  \|A^\frac12g\|_{L_p(\N)}\leq cp\max\{\|\Gamma(g,g)\|_{L_p(\N)},\|\Gamma(g^*,g^*)\|_{L_p(\N)}\}.
 \end{eqnarray*}
for $p>4$ by Lemma \ref{interbr} and Proposition \ref{a.u.}. The
same inequality is proved in \cite{JMr}, Theorem 2.5.13 with
constant $cp$ for $2<p<4$.
\end{proof}
\begin{rem}
Let $\N$ be a commutative von Neumann algebra, for example
$\N=L_\infty({\Bbb R}^n)$. Let $A=\triangle$ on ${\Bbb R}^n$ and $R$
be the classical Riesz transform, i.e. $R(f)=(\cdots,\partial_i
(-\triangle)^{-\frac12}f,\cdots )$. It is well known that $R$ is
$L_p$-bounded uniformly on the dimension $n$ for $1<p<\infty$ (see
\cite{Stfree}, \cite{MeRiesz}, \cite {psr}). For $p=1$, a dimension
free weak $(1,1)$ estimate is due to Varopoulos (see
\cite{voRiesz}). It is desirable to have some results for $p=\infty$
which implies the estimate in the range $1<p<\infty$ by interpolation. Note, in
this case,
$$\|f\|_{BMO(\Gamma)}=\|R(f)\|_{BMO(\partial)}.$$
By Proposition \ref{bmo3}, we have the dimension free estimate
$$\|R(f)\|_{BMO(\partial)}\leq c\|f\|_\infty.$$
\end{rem}

\section{Interpolation}
We will prove an interpolation theorem for BMO spaces associated
with semigroups of operators. Our BMO spaces are then good endpoints
for noncommutative $L_p$ spaces.

Let $(T_t)$ be a standard semigroup on $\N$ admitting a (reversed)
standard Markov dilation $(\M_t,\pi_t, E_t)$. We say the dilation
has {\it a.u. continuous path} if there exist weakly dense subsets
$B_p$ of $L_p(\N)$ such that both $m(f)$ and $n(f)$ have a.u.
continuous path for all $2\leq p< \infty$. Here $m(f)$ and $n(f)$
are martingales given as in  (\ref{ms}) and Proposition
\ref{meyer1}.

 \begin{prop}\label{a.u.}
Suppose a standard semigroup $T_t$ satisfies $\Gamma^2\geq0$ and
admits an a.u. continuous standard (reversed) Markov dilation
$(\pi_t, \M_{t]})$. Then the martingale
$\hat{n}_a(f)=\hat{\pi}_{\ttt_a\wedge s}(Pf)$ in Meyer's model is
a.u. continuous for all $f\in L_p(\N), p>2$.
 \end{prop}

\begin{proof} This is the second part of Lemma 2.5.3 (ii) of \cite{JMr}.
\end{proof}

We use the notation \emph{$L_p^0(\N), 1\leq p\leq\infty$} for the complemented subspace of $L_p(N)$ which is orthogonal to
 \[  ker(A_p)\lel \{f\in\dom_p(A),Af=f\} \lel
 \{f\in L_p(\N), \lim_tT_tf=f\} \pl .\]
Equivalently, $L_p^0(\N)=\{f\in
L_p(\N), \lim_{t\rightarrow \infty} T_tf=0\}$ and hence we may also view $L_p^0(\N)$ as a quotient space. The limit is taken  with respect to the
$\|\cdot\|_{L_p(\N)}$-norm for $1<p<\infty$ and is with respect to
 the weak$^*$ topology for $p=1,\infty$.
Recall from Proposition \ref{kernel} we know that
$\|\cdot\|_{bmo^c(\T)}$ and $\|\cdot\|_{BMO^c(\T)}$ are norms on the
quotient space $\N^0\cup L_2^0(\N)$. Note $Af=0$ implies $T_tf=f$
and $P_tf=f$ for all $t$, we get
$\ker(A_\infty^{\frac12})=\ker(A_\infty)=\{f,\lim_tP_tf=f\}$. So
$\|\cdot\|_{BMO_c(\P)}$ and $\|\cdot\|_{bmo^c(\P)}$ are norms on
$\N^0\cup L_2^0(\N)$ too. The same is true for
$\|\cdot\|_{BMO^c(\Gamma)}$, $\|\cdot\|_{BMO^c(\hat\Gamma)}$, and
for $\|\cdot\|_{BMO^c(\partial)}$.

\subsection{Interpolation in the finite case}
We assume that the underling von Neumann algebra $\N$ is with a
finite trace $\tau$ in this subsection. In this case, all the
BMO-norms associated with semigroups are bigger than the
$L_2(\N)$-norm up to a constant. Let $BMO(\T)$, $BMO(\P)$,
$BMO(\hat{\Gamma})$, $BMO(\Gamma), BMO(\partial)$ and $ bmo(\T),
bmo(\P)$ be the spaces of $f\in L_2^0(\N)$ with finite corresponding
BMO-norms. We consider the complex interpolation couples
$[X,L_p^0(\N)]_{\frac1q}$ with $X$ any of these BMO spaces. See
\cite{bl} for basic properties of complex interpolation method.

\begin{theorem}\label{intpol} Let $(T_t)$ be a standard semigroup of
operators. We have \begin{enumerate} \item[(i)] Assume $(T_t)$
admits a standard Markov  dilation. Then
 $$
L_{pq}^0({\mathcal N})= [X,L^0_p({\mathcal N})]_{\frac1q},$$ with
equivalence constant $\simeq pq$ for all $p\geq1, q>1$ and $X$ being
 semigroup-BMO spaces $BMO(\T)$, $BMO(\P)$,
$BMO(\hat{\Gamma}), BMO(\partial)$, and $bmo(\P)$.

\item[(ii)] Assume $(T_t)$ admits a reversed Markov dilation with  a.u. continuous
path. Then
 $$
L_{pq}^0({\mathcal N})= [bmo(\T),L^0_p({\mathcal N})]_{\frac1q},$$
equivalence constant $\simeq pq$  for all $p\geq1, q>1$. If, in
addition, $T_t$ satisfies $\Gamma^2\geq0$, we have
 $$
L_{pq}^0({\mathcal N})= [BMO(\Gamma),L^0_p({\mathcal
N})]_{\frac1q},$$ equivalence constant $\simeq pq$  for all $p\geq1,
q>1$.
\end{enumerate}
\end{theorem}

\begin{proof} For any choice of $X$, note that the trivial inclusion
${\mathcal N}^0\subset X$ implies
 \[ L_{pq}^0({\mathcal N})\subset [X,L_p^0({\mathcal N})]_{\frac{1}{q}} \pl .\]
Assume a Markov dilation exists, for $X=BMO(\partial)$, we consider
Meyer's model in section 4. Note $(\hat
n_a(x))_0=\hat{E}_0\pi_{\ttt_a}(x)=\pi_0P_ax$ for all $x\in
L_2(\N)\supseteq X$. According to Lemma \ref{bmo2} (ii), we get that
$P_{br}\pi_{\tau_a}$ embeds $BMO(\partial)$ into $BMO(\hat{\M}_a)$.
Thus $P_{br}\pi_{\tau_a}$ embeds
$[BMO(\partial),L_p^0(\N)]_{\frac1q}$ into $
[BMO(\hat\M_a),L_p(\hat\M_a)]_{\frac1q}$ because it embeds $L_p(\N)$
into $L_p(\hat\M_a)$ too. Note
 \[ [BMO(\hat\M_a),L_p(\hat\M_a)]_{\frac1q}\lel L_{pq}(\hat\M_a)\]
with equivalence constant $\simeq pq$ by Lemma \ref{JuPe}. We deduce
that,
$$\|P_{br}\pi_{\tau_a}x\|_{pq}\leq cpq \|x\|_{[BMO(\partial),L_p^0(\N)]_{\frac1q}} $$
holds for all $x\in [BMO(\partial), L_p^0(\N)]_{\frac1q}$. By Lemma
2.5.11 of \cite{JMr}, we have (note $\rho_a$ there denotes
$\pi_{\tau_a}$ and $Pr$ is the projection to $L_p^0(\N)$)
$$\|x\|_{pq}\leq 2\lim_{a\rightarrow \infty}\|P_{br}\pi_{\tau_a}x\|_{pq}\leq cpq \|x\|_{[BMO(\partial),L_p^0(\N)]_{\frac1q}}.$$
We obtain the desired result for $X=BMO(\partial)$.
 For $X=BMO(\P), bmo(\P), BMO(\hat\Gamma), BMO(\T)$, the interpolation result follows from the relation
that $\N^0\subset X\subset BMO(\partial)$ because of Theorem
\ref{lem0} and Proposition \ref{bmo3} (ii).

We now prove (ii). Assume the admitted Markov dilation has a.u.$\!$
continuous path. By Proposition \ref{bmo1}, we see that $\pi_0$
embeds $bmo(\T)$ into $bmo(\M)$. Now, for any $x\in bmo(\T)\subset
L_2(\N)$, we can find a net $x_\la\in L_2(\N)\cap \N$ converging to
$x$ in $L_2(\N)$. So $\pi_0(x_\la)\in L_2(\M)\cap \M$ converging to
$\pi_0(x)$ in $L_2(\M)$. By Lemma \ref{ctmar}, $\pi_0(x)\in BMO(\M)$
and $\|\pi_0(x)\|_{BMO(\M)}\leq c\|x\|_{bmo(\M)}$. Therefore,
$\pi_0$ embeds $bmo(\T)$ into $BMO(\M)$. By the same argument used
for the proof of (i), we obtain the desired result.

We now turn to $BMO(\Gamma)$,  Lemma \ref{bmo2} (iii) implies that
$P_{\Gamma}\pi_{\tau_a}$ embeds $BMO(\Gamma)$ into $bmo(\hat\M_a)$.
Note that Proposition \ref{a.u.} implies the a.u. continuity of
$\hat n_a(x)=(\hat E_t\pi_{\tau_a}x)_t$ for all $x\in L_{2}(\N)\cap
\N$ assuming $\Gamma^2\gl 0$.  Then $P_{\Gamma}\pi_{\tau_a}$ embeds
 $BMO(\Gamma)$ into $BMO(\hat\M_a)$ by Lemma \ref{ctmar} and the argument used for (ii).
 Repeat
 the argument used for the proof of (i), we obtain (iii).
\end{proof}

\begin{rem} According to \cite{JRS} we have a Markov dilation for finite von Neumann algebras. Hence $BMO(\partial)$ solves  problem \eqref{mainpb} in this case.
\end{rem}

As a consequence, we obtain the boundedness of Fourier multiplier
 $M_a$ discussed in Section 3.
\begin{cor}\label{Lp} Let  $(T_t)$ be a standard semigroup admitting a Markov
dilation. Let $M_a$ be as in Section 2. Then
\begin{eqnarray}
\|M_af\|_{L_p(\N)}&\leq& c_p\|f\|_{L_p(\N)},
\end{eqnarray}
with $c_p$ in order of $\simeq \max\{p,\frac1{p-1}\}$. In
particular, for $M_a=L^{is}$, we have \begin{eqnarray}\label{Lises}
\|L^{is}f\|_{L_p(\N)}\leq c_{s,p}\|f\|_{L_p(\N)}, \end{eqnarray}
with
$$c_{s,p}\simeq\max\{p,\frac1{p-1}\}|s|^{-|\frac12-\frac1p|}\exp(|\frac
{\pi s}2-\frac{\pi s}p|).$$
\end{cor}
\begin{proof}
Apply Theorem \ref{mul1} and Theorem \ref{intpol} to $M_a$ and their
adjoint operators, we have for $f\in L_p^0({\mathcal N})$,
$$\|M_af\|_{L_p(\N)}\leq c\max\{p,\frac1{p-1}\}\|f\|_{L_p(\N)}
$$
for all $1<p<\infty$. Since $M_a$'s vanish on $f$ with $\lim_t
T_tf=f$, they are bounded on the whole $L_p(\N)$. For $M_a=L^{is}$,
we have \begin{eqnarray*} \|L^{is}f\|_{L_2(\N)}&\leq&
\|f\|_{L_p(\N)},\\
\|L^{is}f\|_{BMO(\partial)}&\leq&
c\Gamma(1-is)^{-1}\|f\|_{BMO(\partial)}.
\end{eqnarray*}
By interpolation, we have, for all $1<p<\infty$,
\begin{eqnarray*} \|L^{is}f\|_{L_p(\N)}\leq c\max\{p,\frac1{p-1}\}\Gamma(1-is)^{-|1-\frac 2p|}
\|f\|_{L_p(\N)}.
\end{eqnarray*}
It is well known that,  e.g. see page 151 of \cite{Ti},
\begin{eqnarray}\label{Gaes} |\Gamma(1-is)|\simeq |s|^\frac12e^{-\frac
{\pi|s|}2}.\end{eqnarray} Therefore, we conclude,
\begin{eqnarray*}
\|L^{is}f\|_{L_p(\N)}\leq
c\max\{p,\frac1{p-1}\}|s|^{-|\frac12-\frac1p|}e^{|\frac{\pi
s}2-\frac{\pi s}p|} \|f\|_{L_p(\N)}.
\end{eqnarray*}
\end{proof}
\begin{rem}
It is known that standard semigroups $(T_t)$ on von Neumann algebras
$VN(G)$ of a discrete group always admit a Markov dilation (see
\cite{Ri}). Moreover, a recent result of Junge/Ricard /Shlyakhtenko
(see \cite{JRS}) shows that standard semigroups $(T_t)$ on any
finite von Neumann algebras admits a Markov dilation and for the bounded generators $A_t=t^{-1}(I-T_t)$ the Markov dilations also has  almost uniformly continuous path.
\end{rem}

\begin{rem} \label{better}The $L_p$-boundedness of Fourier multipliers $M_a$ could be proved directly  following
E. Stein's Littlewood-paley $g$-function technique (see
\cite{stein-erg}) by the noncommutative $H_p$ theory developed in
\cite{JLX}, with worse constants. It could be also obtained
following a classical argument of M. Cowling (see \cite{Cow})
through `transference technique' in the noncommutative setting,
which could become available after \cite{JRS}. However,
`transference technique' does not seem to work for $BMO$. Cowling
did obtain optimal $L_p$-boundedness constants for the imaginary
powers $L^{is}$ on abelian groups, although  our method provides a
slightly better estimate on $s$ as $s\rightarrow \infty$ (see
(\ref{Lises}) ). But Cowling did not have optimal $L_p$-boundedness
constants for general multipliers $M_a$'s (he had $\simeq
\max\{p^\frac52,(p-1)^{-\frac52}\}$, see Theorem 3 of \cite{Cow}).
\end{rem}

As another application,  we obtain optimal constants for the
 noncommutative maximal ergodic inequality proved by Junge and Xu
 (see Theorem 5.1 and Corollary 5.11 of \cite{JX3}).

\begin{cor} Suppose $(T_t)$ is a standard semigroup admitting a
Markov dilation, then
\begin{eqnarray}
\|\sup_tT_tf\|_{L_p(\N)}\leq
c\max\{1,\frac1{(p-1)^2}\}\|f\|_{L_p(\N)}.
\end{eqnarray}
\end{cor}

\begin{proof} The proof is to write $T_t-\frac1t\int_0^tT_vdv$ as an weighted average of $L^{is}$ for each $t$ as Cowling did (see \cite{Cow})
 and use the uniform estimate of $L_p(\N)$-boundedness of
$L^{is}$. From the elementary identity
$$\frac1\pi\int_0^{+\infty}
\lambda^{is}\Gamma(1-is)(1+is)^{-1}ds=e^{-\lambda}-
\int_0^1e^{-u\lambda}du,$$ we deduce by functional calculus that
\begin{eqnarray}\label{id}
\frac1\pi\int_0^{+\infty}
(tL)^{is}\Gamma(1-is)(1+is)^{-1}ds=e^{-tL}-\int_0^1e^{-utL}du=T_t-\frac1t\int_0^tT_{v}dv.
\end{eqnarray}
Theorem 4.1 and Theorem 4.5 of \cite{JX3} imply that
\begin{eqnarray}\label{Mt}
\|\sup_t\frac1t\int_0^tT_{v}fdv\|_{L_p(\N)}\leq c
\max\{p,\frac1{(p-1)^2}\}\|f\|_{L_p(\N)}.\end{eqnarray} On the other
hand, for any $a\in L_q(\N_+,L_1(0,\infty))$, $\frac1p+\frac1q=1$,
we have
\begin{eqnarray}
&&|\tau \int_0^\infty a(t)\int_0^{+\infty}
(tL)^{is}(f)\Gamma(1-is)(1+is)^{-1}dsdt|\nonumber\\
&=&|\tau \int_0^{+\infty}\int_0^\infty a(t)t^{is}dt
L^{is}(f)\Gamma(1-is)(1+is)^{-1}ds|\nonumber\\
&\leq&\sup_s\|\int_0^\infty a(t)t^{is}dt\|_{L_q(\N)}\int_0^{+\infty}
\|L^{is}(f)\|_{L_p(\N)}|\Gamma(1-is)(1+is)^{-1}|ds. \label{Tt-Mt}
\end{eqnarray}
A combination of (\ref{Lises}), (\ref{Gaes}),  (\ref{id}), and
(\ref{Tt-Mt}) implies that
\begin{eqnarray}\label{5.9}
|\tau \int_0^\infty a(t)(T_t-\frac1t\int_0^tT_vdv)dt |\leq
c\max\{p,\frac1{p-1}\}\|\int_0^\infty
a(t)dt\|_{L_q(\N)}\|f\|_{L_p(\N)}.
\end{eqnarray}
Without loss of generality, assume $f\geq0$. We deduce by duality
(see Proposition 2.1 (iii) of \cite{JX3}) that,
\begin{eqnarray*}
&&\|\sup_t T_tf\|_{L_p(\N)}\\
&\leq& \sup_{a\in L_q(\N_+,L_1(0,\infty)),
\|a\|\leq 1}\tau \int_0^\infty a(t) T_tfdt,\nonumber\\
&=&\sup_{a\in L_q(\N_+,L_1(0,\infty)), \|a\|\leq 1}\tau
\int_0^\infty a(t)
(T_tf-\frac1t\int_0^tT_vfdv+\frac1t\int_0^tT_vfdv)dt,\nonumber\\
&\leq&\sup_{a\in L_q(\N_+,L_1(0,\infty)), \|a\|\leq 1}\tau
\int_0^\infty a(t)
(T_tf-\frac1t\int_0^tT_vfdv)dt+\|\sup_t\frac1t\int_0^tT_vfdvdt\|_{L_p(\N)}.
\end{eqnarray*}
By (\ref{Mt}) and (\ref{5.9}) we obtain,
\begin{eqnarray*}
\|\sup_t T_tf\|_{L_p(\N)}\leq
c\max\{p^2,\frac1{(p-1)^2}\}\|f\|_{L_p(\N)}.
\end{eqnarray*}
 Note
\begin{eqnarray*}
\|\sup_t T_tf\|_{L_\infty(\N)}\leq
\|f\|_{L_\infty(\N)}.\end{eqnarray*} Apply the interpolation result
of Theorem 3.1 of \cite{JX3}, we obtain
\begin{eqnarray*}
\|\sup_t T_tf\|_{L_p(\N)}&\leq&
c\max\{1,\frac1{(p-1)^2}\}\|f\|_{L_p(\N)},
\end{eqnarray*}for all $1<p<\infty$.
\end{proof}

\subsection{Interpolation in the semifinite case}${\atop }$

We will extend Theorem \ref{intpol} to the case that the underling
von Neumann algebras $\N$ is semifinite. In this case, $BMO$ is no
longer a subspace of $L_2$. To study the interpolation result, we
first have to obtain a larger space that the interpolation couple
$BMO,L_p$ belongs to.

\subsubsection{$L_p$-Hilbert module} We will need the following definition and lemma of $L_p$-Hilbert  module due to Junge/Sherman (see
\cite{JS}). For $p=\infty$ these spaces are well-known through the
GNS construction for a completely positive map (see Paschke \cite{Pa} and Lance, \cite{La}, Corollary 6.3).

\begin{defi}
\label{module} Let $\M$ be a semifinite von Neumann algebra. Let $E$
be an $\mathcal{M}$ right module with an $L_{\frac
p2}(\mathcal{M})$-valued inner product $\langle \cdot ,\cdot \rangle
$.  A (right) Hilbert $L_p(\mathcal{M})$ $(1\leq p<\infty )$ module,
denoted by $L_p^c(E)$, is the completion of $E$ with respect to the
norm $||\cdot ||=\|\langle \cdot ,\cdot \rangle \|_{L^{\frac
p2}(\mathcal{M})}^{\frac 12}.$ A (right) Hilbert $L_\infty
(\mathcal{M})$ module, denoted by $L_\infty^c(E)$ is the completion
of $E$ with respect to the strong operator topology, briefly STOP
topology. The STOP topology is induced by the family of seminorms
$\| x\|_\xi=\tau (\xi\langle x ,x \rangle )]^{\frac 12}$.
\end{defi}

Here is an easy proposition which we will use frequently.
\begin{prop}
\label{converge} Suppose $(L_\infty^c(E),\langle\cdot,\cdot\rangle)$
is a Hilbert $L_\infty(\M)$-module. Suppose a net $x_\la \in \M$
converges to $x\in L_\infty(E)$ in the STOP topology. Then $\langle
x_\la,x_\la\rangle$ weak$^*$ converges in $\M$. We denote the limit
by $\langle x,x\rangle$.
\end{prop}

Given a Hilbert space $H$, denote by $B(H)$ the space of all bounded
operators on $H$. Choose a norm one element $e\in H$, let $P_e$ be
the rank one projection onto Span$\{e\}.$ For $0<p\leq\infty$, let
$$L^p(\mathcal{M},H_c)=L_p(B(H)\otimes
\mathcal{M}))(1\otimes P_e).$$ Namely, $L^p(\mathcal{M},H^c)$ is the
column subspace of $L^p(B(H)\otimes \mathcal{M}))$ consisting of all
elements with the form $x(1\otimes P_e)$ for $x\in L^p(B(H)\otimes
\mathcal{M}))$. The definition of $L^p(\mathcal{M},H^c)$ does not
depend on the choice of $e$. $L^p(\mathcal{M},H^c)$ can be
identified
as the predual of $L^q(\mathcal{M},H_c)$ with $q=\frac p{p-1}$ for $%
1\leq p<\infty$. The reader can find more information on
$L^p(\mathcal{M},H_c)$ in Chapter 2 of \cite{JLX}.

\begin{lemma}
\label{lance} $L_p^c(E)$ is isomorphic to a complemented subspace of
$L^p(\mathcal{M},H_c)$ for some Hilbert space $H$. Moreover, the
isomorphism does not depends on $p$ and
 \begin{equation}
 \label{moddual} (L_p^{c}(E))^* \lel L_{q}^c(E)
 \pl ,
 \end{equation}
 for all $1\leq p<\infty, \frac1p+\frac1q=1$.
Here the anti-linear duality bracket $(w,z)=tr(\langle w,z\rangle)$
is used.
\end{lemma}

\subsubsection{Interpolation for $BMO(\partial), BMO(\T), BMO(\P), BMO(\hat\Gamma)$ and $bmo(\P)$.}
We use Meyer's model to prove an interpolation result for the BMO
space corresponding to the $\|\cdot\|_{bmo(\partial)}$-norm. For
$x\in \N^0$ we recall the definition
 \[ \|x\|_{bmo^{c}(\partial)}\lel \sup_t \|\int_0^t  P_{t}|P'_s(x)|^2s
 ds\|^{\frac12}
 \simeq \sup_t \|\int_0^\infty  P_{s+t}|P'_s(x)|^2\min(t,s)
 ds\|^{\frac12} \pl .\]
 Define the $L^{\infty}({\Bbb R}_+)\ten \N$-valued inner product on $\N\ten \N$ by
 $$\langle x\ten a, y\ten b\rangle_{\partial}= a^*(\int_0^\infty P_{s+t}(P'_s(x^*)P'_s(y))\min(t,s)ds)b.
$$
Let $V$ be the Hilbert $L^\infty$-module corresponding to this inner
product. Let $BMO^c(\partial)$ be the strong operator closure of
$\N^0$ in $V$ via the embedding
$$\Phi:x\rightarrow x\ten 1.$$
Let $BMO^r(\partial)$ be the strong operator closure of $\N^0$ in
$V$ via the embedding $x\rightarrow x^*\ten 1$.

We define the column and row space of $BMO(\T), BMO(\P)$,
$BMO(\hat\Gamma)$ and $bmo(\P)$ similarly by using Hilbert
$L_\infty$-modules corresponding to respective BMO-norms given in
Section 2.

To understand the intersection of $BMO^c$ and $BMO^r$, we need the
following observation.

\begin{lemma}\label{vectorspace} Let $x\in X$ with $X\in\{BMO^c(\partial), BMO^c(\T), BMO^c(\P)$,  $BMO^c(\hat\Gamma), bmo^c(\P)$\}. Then $P_t'x$ exists in $\N$ and
\begin{eqnarray}\label{X}
 \|P_t'x\|_{\infty}\kl Ct^{-1} \|x\|_{X} \pl .\end{eqnarray}
\end{lemma}

\begin{proof}
Fix $t>0$. Let $x_\lambda\in \N^0\subset BMO^c(\partial)$ be a net
such that $\Phi(x_\lambda)=x_\lambda\ten 1$ converges in $V$ with
respect to the STOP topology. We will show that $P_t'x_\la$ weakly
converges in $\N$ and the limit (denoted by $P_t'x$) has norm
smaller than $ct^{-1}\|x\|_{BMO^c(\partial)}$. This is what we mean
by $P_t'x$ exists in $\N$.

We first deduce from Proposition \ref{monot} that, for $t>0$,
 \begin{align}
  \frac{t^2}{2} |\frac{\partial P_{2t}}{2t}x_\lambda|^2 &\le \int_0^t |\frac{\partial P_{2t}}{2t}x_\lambda|^2
  sds \lel  \int_0^t |P_{2t-s}P'_{s}x_\lambda|^2   sds \nonumber\\
  &\le \int_0^t P_{2t-s}(|P'_{s}x_\lambda|^2)   sds
  \kl  \int_0^t \frac{2t-s}{t+s}P_{t+s}(|P'_{s}x_\lambda|^2)   sds \nonumber\\
  &\kl  2 \int_0^{\infty} P_{t+s}(|P'_sx_\lambda|^2) \min(s,t) ds\\
  &= 2\langle \Phi(x_\la),\Phi(x_\la)\rangle_\partial\label{ts}
  \pl .
  \end{align}
By Proposition \ref{converge}, we know that $P_t'x_\lambda$
converges with respect to the strong operator topology of $\N$ and
the limit exists in $\N$ with a norm bounded by $\frac
ct\|x\|_{BMO^c(\partial)}$, since $\Phi(x_\lambda)$ converges in the
STOP topology. Note the $\|\cdot\|_{BMO^c(\partial)}$-norm is
smaller than any of the other $X$-norms by Lemma \ref{bmo3} (ii) and
Theorem \ref{lem0}. We obtain (\ref{X}) for all $X$.\end{proof}

We say that $x\in BMO^c(\partial)$ belongs to $BMO^r(\partial)$ if
$P_t'x=P_t'y$ for some $y\in BMO^r(\partial)$ for all $t>0$. This
$y$ is unique in $BMO^r(\partial)$. In fact, assume there are two
weak$^*$ convergent nets $y_\lambda,\tilde y_\lambda$ in
$BMO^r(\partial)$ such that $P_t'x=P_t'y=P_t'\tilde y$ holds for the
limit elements $y,\tilde y\in BMO^r(\partial)$ and any $t>0$. Then
$P_t'(y_\lambda-\tilde y_\lambda)$ converges to $0$ for any $t$ with
respect to the weak$^*$ topology of $\N$. Hence $\int_0^\infty
P_{b+s}|P_s'(y_\lambda-\tilde y_\lambda)^*|^2sds$ weak $^*$
converges to $0$ for any $b$ by the dominated convergence theorem.
This means $y-\tilde y=0$ in $BMO^r(\partial)$. Set $BMO(\partial)$
to be the space consisting of all such $x$'s equipped with the
maximum norm
$$
\|x\|_{BMO(\partial)}=\max\{\|x\|_{BMO^c(\partial)},\|y\|_{BMO^r(\partial)}\},$$
Here $y$ is the unique $y\in BMO^r(\partial)$ such that
$P_t'x=P_t'y$ for all $t>0$ as we explained above. Define $BMO(\T),
BMO(\P)$, $ BMO(\hat\Gamma)$ and $bmo(\P)$ to be the intersection of
the corresponding row, column spaces similarly.

Once we have these definitions, the same proof of Theorem
\ref{intpol} implies
\begin{theorem} Let $1\le p<\infty$. Assume a standard semigroup $T_t$ admits a standard Markov dilation. Then
 \[ [X,L_1^0(\N)]_{\frac{1}{p}}\lel L_p^0(\N) \pl ,\]
 with equivalence constant in order  $p$ for
 $X=BMO(\partial),BMO(\T), BMO(\P)$, $BMO(\hat\Gamma)$ or $bmo(\P)$.
\end{theorem}

\subsubsection{Interpolation for $bmo(\T)$.}
For the interpolation for $bmo(\T)$, besides an appropriate
definition of the interpolation couple $bmo(\T), L_1$, we also need
to show that $L_p^0(\N)$ is dense in $[bmo(\T),
L^0_1(\N)]_{\frac1p}$ because we only assume that the special
martingales $m(x)$'s have a.u. continuous path and general
martingales may not, while in the case of $BMO(\partial)$ we have
automatically that all Brownian martingales have continuous path.
This difficulty already appeared in the finite case (see the end of
the proof of Theorem \ref{intpol}). We will go around it by defining
an abstract predual of $bmo(\T)$.

For a standard semigroup $\T=(T_t)$ on $\N$. We consider the
$L_{\frac p2}(\ell_\infty({\Bbb R}_+)\ten \N)$- valued inner
products on $E=\ell_\infty({\Bbb R}_+)\ten (\N\ten \N)$,
\begin{eqnarray*}
 \langle a\ten b, c\ten d\rangle^c_{T}
 \lel b_t^*T_t(a_t^*c_t)d_t,
 \langle a\ten b, c\ten d\rangle^r_{T}
 \lel b_tT_t(a_tc^*_t)d^*_t,
 \end{eqnarray*}
for $a\ten b\in \ell_\infty({\Bbb R}_+)\ten (\N\ten \N)$.
Denote by $V_p^{c}$ (resp. $V_p^r$)
the $L_p(L_\infty({\Bbb R}_+)\ten \N)$-Hilbert module corresponding to
$E,\langle\cdot,\cdot\rangle^c_{\T}$ (resp. $E,\langle\cdot,\cdot\rangle^r_{\T}$).

 Let us denote by $w :\N\to
E $ the embedding map $w(x)_t=x\ten 1-1\ten T_tx$. Then
\begin{eqnarray*}
\langle w(x),w(x)\rangle^c_{\T}&=&T_t|x|^2-|T_tx|^2. \\
\langle w(x),w(x)\rangle^r_{\T}&=&T_t|x^*|^2-|T_tx^*|^2.
\end{eqnarray*}
Denote by $w_c^*$ (resp. $w_r^*$) the adjoint of $w$ with respect to
$\N, \tau(x^*,y); E,\langle\cdot,\cdot\rangle^c_{\T}$ (resp. $\N,$
$\tau(xy^*); E,\langle\cdot,\cdot\rangle^r_{\T}$). We have
 \begin{equation} \label{wtform}
  w_c^*(a\ten b) \lel \sum_t a_tT_t(b_t)-T_t(T_t(a_t)b_t) \pl,
  w_r^*(a^*\ten b^*) \lel \sum_t T_t(b^*_t)a^*_t-T_t(b_t^*T_t(a^*_t)),
 \end{equation}
 for $a\ten b\in \ell_1({\Bbb R}_+)\ten (\N\ten\N).$
Indeed, for $x\in \N$ and $z=a\ten b=(a_t\ten b_t)_t$,
 \begin{align*}
 \tau(x^*w_c^*(z))&= \tau\sum_t(\langle x\ten 1-1\ten T_tx,a_t\ten b_t\rangle_{\T}^c)
  \lel \tau\sum_t ( T_t(x^*a_t)b_t)-tr(T_t(x^*)T_t(a_t)b_t)
  \\
 &=\tau\sum_t   (x^*(a_tT_t(b_t)-T_t(T_t(a_t)b_t))) \pl .
  \end{align*}

\begin{defi}
 \begin{enumerate}
  \item[(i)] The space $bmo^c(\T)$ (resp. $bmo^r(\T)$) is defined as the weak$^*$-closure of
 $\N^0$ in $V_\infty^c$ (resp. $V_\infty^r$) via the embedding $w$.
 \item[(ii)]
 $h_1^c(\T)$ (resp. $h_1^r(\T)$) is defined as the quotient of $V_1^c$
 (resp. $V_1^r$) by the kernel of $w_c^*$ (resp. $w_r^*$).
 The Hardy space $h_1(\T)$ is defined as $h_1^c(\T)+h_1^r(\T)\subset
 L_1(\N)$. More precisely, for $f\in  L_1(\N)$,
 $$ \|f\|_{h_1^c(\T)}=\inf\{\|v\|_{V_1^c},w^*_c(v)=f\}.$$
  \end{enumerate}
\end{defi}
In the following Lemma we report some elementary properties.

\begin{lemma}\label{elemprop}
 \begin{enumerate}
 \item[(i)]   $x\in h_1^c(\T)$ iff $x^*\in h_1^r(\T)$.
 \item[(ii)] $h_1^c(\T)\cap h_1^r(\T)\cap L_p^0(\N)$ is dense in $L_p^0(\N)$ for $1\leq p<\infty$.
  \item[(iii)]  $h_1(\T)\cap L_p$ is dense in $h_1(\T)$ for all $1\leq p\leq\infty$.
 \item[(iv)] $(h_1^c(\T))^*\lel bmo^c(\T)$, $(h_1^r(\T))^*\lel bmo^r(\T)$. Assume $h_1^c(\T)\cap h_1^r(\T)$ is dense in both
 $h_1^c(\T)$ and $h_1^r$. Then $(h_1(\T))^*=bmo(\T)=bmo^c(\T)\cap bmo^r(\T)$.
\item[(v)] Assume that $(T_t)$ admits a reversed
Markov dilation $\M_t, \pi_t$.
  Then the homomorphism $\pi_0:\N^0\to  bmo^c(\M)$ extends to a
 weakly continuous map on $bmo^c(\T)$ and $h_1^c(\T)\cap h_1^r(\T)$ is dense in both
 $h_1^c(\T)$ and $h_1^r(\T)$.
 \end{enumerate}
\end{lemma}

\begin{proof} (i) is obvious because $a\ten b\in V_1^c$ iff $a^*\ten b^*\in V_1^r$ and
$w$ is bounded and injective from  $\N^0$ to $V_\infty^c\cap V_\infty^r$.
For the proof of
(ii), we first show that
 \[ \{ aT_t(b)-T_{t}(T_t(a)b) : t>0,a , b \in L_2\cap L_{\infty} \}
 \subset L_p^0(\N) \]
is dense in $L_p^0(\N)$. Indeed,
let $y\in L_{p'}(\N)$ such that
 \begin{eqnarray}\label{T(T)}
  tr(aT_t(b)y) \lel tr(T_t(T_t(a)b)y) \end{eqnarray}
holds for all $a,b$ as above. By approximation with support
projections and the weak continuity of $T_t$, we deduce from
(\ref{T(T)}) and the self adjointness property of $T_t$ that
 \[ tr(\tilde{a}y)\lel \lim_{\la,\mu }
 tr(\tilde{a}e_{\mu}T_t(e_{\la})y)
 \lel \lim_{\la,\mu }
 tr(T_t(\tilde{a}e_{\mu})e_{\la}T_ty) \lel tr(\tilde{a}T_{2t}y) \pl
 .\]
This shows $T_{2t}(y)=y$ and hence $y\in
L_{p'}(\N_0)=\overline{(\ker A)^\bot}^{\|\pl\|_p}$. Hence
$L_p^0(\N)\cap h_1^c$ is dense in $L_p^0(\N)$. Similarly,
$L_p^0(\N)\cap h_1^r$ is dense in $L_p^0(\N)$.

For (iii), Let $A$ be the set of $a\ten b=a(t)\ten b(t)$ with $a(t),
b(t)\in L_1(\N)\cap\N$ for all $t$ and $a(t)=b(t)=0$ except finite
many $t$'s. Then $w_c^*(A), w_r^*(A)\in L_p(\N)$ and $A$ is dense in
$V_1^c$ and is dense in $V_1^r$.  We conclude that $h_1^c(\T)\cap
L_p(\N)$ is dense in $h_1^c(\T)$ and $h_1^r(\T)\cap L_p(\N)$ is
dense in $h_1^r(\T)$. So $h_1(\T)\bigcap L_p(\N)$ is dense in
$h_1(\T)$.

For the proof of (iv) we see that the inclusion map
$\iota:h_1^c\to  L_1(\N)$ is injective. By the Hahn Banach theorem,
we deduce that $\iota^*(\N)\subset (h_1^c)^*$ is weakly dense.
However, by definition $h_1^c$ is a quotient of $V_1^c$. Hence $(h_1^c)^*$ is a subspace of
$V_\infty^c.$ We then deduce from \eqref{wtform} that, when restricted to $\N$,
the map $\iota^*$ is given by $\iota^*(x)(t)\lel x\ten 1-1\ten
T_tx$. Thus we have
 \[ (h_1^c)^*\lel \iota(\N) \lel bmo^c \pl .\]
Taking adjoints we get $(h_1^r)^*=bmo^r$. Since $X=h_1^c\cap h_1^r$
is dense in both spaces, we may then embed $bmo^c$ and $bmo^r$ in
$X^*$. We see that the inclusion map $\iota_X:X\to L_1(\N)$ is
injective and factors through the inclusion map $\iota_{h_1}:h_1\to
L_1(\N)$. Since $X\subset h_1$ is dense, we deduce that $h_1^*$ is
the weak$^*$-closure of
 \[ h_1^* \lel \overline{\iota^*(\N)}^{\si(h_1^*,h_1)}
  \subset X^* \pl .\]
Note that the last inclusion is injective and certainly
$h_1^*\subset bmo^c\cap bmo^r$ because elements in $\N$ give rise to
functionals which coincide on the intersection. For the converse
inclusion $bmo^c\cap bmo^r\subset h_1^*$, it suffices to recall that
a bounded functional extends uniquely from a dense subspace.

We now prove (v). Recall that a net $x_\lambda\in \N^0$ weakly
converges in $bmo^c(\T)$, if the inner product  $\langle
w(x_\lambda),w(x_\lambda)\rangle_\T^c=T_t|x|^2-|T_tx|^2$ weakly
converges in $\ell_\infty\ten\N$. This is equivalent to the weak
convergence of $\langle
\pi_0x_\lambda),\pi_0x_\lambda)\rangle_\E^c=\pi_t(T_t|x|^2-|T_tx|^2)$
in $\ell_\infty\ten \M$, which is the meaning of weak$^*$
convergence of $(\pi_0x_\lambda)$ in $bmo^c(\M)$ (see \cite{JuPe}).
Therefore, $\pi_0(bmo^c(\T))\subset bmo^c(\M)$ is a weak$^*$ closed
subspace and $\pi_0^*(h_1^c(\M))=h_1^c(\T)$. We obtain the density
of $h_1^r(\T)\cap h_1^r(\T)\subset h_1^c(\T)$ by the corresponding
result on martingale Hardy spaces.\qd

\begin{lemma}\label{swindle} Assume that a standard semigroup $(T_t)$ has
a reversed Markov dilation with a.u. continuous path. Then
 \[ \pi^*_0(H_1^c(\M)) \subset h_1^c(\T) \subset L_1(\N) \pl .\]
\end{lemma}
\begin{proof} We have seen  that $\pi^*_0(h_1^c(\M))=
h_1^c(\T)$ and $\pi^*_0(H_1^c(\M))\subset \pi^*_0(L_1(\M))=L_1(\N)$.
 Let us recall that $H_1^c(\M)=h_1^c(\M)+h_1^d(\M)$. We are going to show that $\pi^*_0(h_1^d(\M))$ vanishes in $L_1(\N)$.
 By density it suffices to consider $\xi\in h_1^d(\M)\cap h_p^{d}(\M)$ for some $1<p<2$.
Recall that there are weakly dense subsets  $B_{q}$ of $L_{q}(\N)$
such that the martingale $m(f)=(E_{[t}(\pi_0f))_t$ has \au
continuous path if $T_t$ admits a reversed Markov dilation with a.u.
continuous path. (see the definition at the beginning of this
section). Let $y\in B_{q}$. By Lemma \ref{ctmar},
 \[ \|\pi_0(y)\|_{h_{q}^d} \lel 0 \pl .\]
This implies
 \begin{align*}
  |tr(\pi_0^*(\xi^*)y)| &= |tr(\xi^*\pi_0(y))| \kl
  \lim_{\si} \|\xi\|_{h_{p}^d(\si)}
  \|\pi_0(y))\|_{h_{q}^d(\si)} \lel 0 \pl .
  \end{align*}
Hence  $tr(\pi_0^*(\xi)\cdot)$ vanishes on a weakly dense set of
$L_q(\N)$ and is $0$ in $L_p(\N)$. So it is $0$ in $L_1(\N)$. Thus
$\pi_0^*$ is $0$ on $h_1^d\cap h_p^d$ and hence identically $0$.
Therefore we have indeed $\pi^*_0(H_1^c(\M))\subset h_1^c(\T)$.
 \qd

\begin{theorem}\label{intt1} Let $1<p<\infty$ and $(T_t)$ be a standard semigroup admitting a reversed Markov dilation with \au continuous path.
Then
 \[ [bmo^0(\T),L_1^0(\N)]_{\frac1p}\lel [bmo^0(\T),h_1(\T)]_{\frac1p}
 \lel [\N^0,h_1(\T)]_{\frac1p} \lel  L_p^0(\N) \pl .\]
\end{theorem}

\begin{proof} By Lemma \ref{swindle}, we have, for $1<p\le 2$ and
$\frac{1}{p}=\frac{1+\theta}{2}$,
$$L_p^0(\N)=\pi_0^*(\pi_0L_p^0(\N))\subset \pi_0^*(L_p^0(\M))\subset \pi_0^*[L_2^0(\M),H_1(\M)]_\theta\subset [L_2^0(\N),h_1(\T)]_\theta.$$
Combining this with the trivial inclusion
 \[  [L_2^0(\N),h_1(\T)]_\theta\subset
 [L_2^0(\N),L_1^0(N)]_\theta \lel L_p^0(\N)\]
we equality in this range.  Theorem \ref{intt1} follows by duality and Wolffs' theorem (see \cite{Mu} for a similar
argument). \qd

\subsubsection{Interpolation for $BMO(\Gamma)$.}
Our last concern in this section are interpolation result for
$BMO(\Gamma)$ spaces. We first need some definitions. We define
 a $\ell^{\infty}({\Bbb R}_+)\ten \N$-valued inner product on $\N\ten \N$ by
 \[ \langle x\ten a,y\ten b\rangle_\Gamma \lel a^*\int_0^{\infty} P_{s+t}\Gamma(P_sx,P_sy)
 \min(s,b) ds b\pl .\]
Let $\mathcal{L}$ be the Hilbert $\ell^{\infty}({\Bbb R}_+)\ten
\N$-module corresponding to this inner product. Recall that we
denote by $P_{\Gamma}$ the projection on the spatial part in Meyer's
model.

\begin{defi} Let $BMO^c(\Gamma)$ be the weak$^*$-closure of $\N^0$ in
$\mathcal{L}$ via the embedding
$$\Phi: x\rightarrow x\ten 1.$$
Let
$BMO^r(\Gamma)$ be the weak$^*$-closure of $\N^0$ in $\mathcal{L}$
via the embedding $x\rightarrow x^*\ten 1$.
\end{defi}

 Let  $x_\lambda \in \N^0 $ be a bounded net in $BMO^c(\Gamma)$ which weak$^*$ converges to $x\in BMO^c(\Gamma)$. Recall that $P_bx$ exists in $BMO^c(\Gamma)$ for any $b>0$ and $$S(t)=\int_0^\infty P_{t+s} {\Gamma }[P_s(x)]\min\{t,s\}ds.$$
exists in $\ell_\infty({\Bbb R}_+)\otimes \N$ as the weak$^*$ limit of
$$S_\lambda(t)=\int_0^\infty P_{t+s} {\Gamma }[P_s(x_\lambda)]\min\{t,s\}ds
.$$ Here and in the following, ${\Gamma }[x]$ denotes ${\Gamma }(x,x)$ for simplification. We need the following lemma to understand the intersection of
$BMO^c(\Gamma)$ and $BMO^r(\Gamma)$.

\begin{lemma}\label{lastlemma} Let $(T_t)$ be a standard semigroup satisfying $\Gamma^2\gl
0$. Then,  for any $x\in BMO^c(\Gamma)$

(i) $P_{2b}\Gamma (P_bx,P_bx)$ exists in $\N$ for any $b>0$, and
\begin{eqnarray*}
\|P_{2b}\Gamma (P_bx,P_bx)\|\leq \frac {6}{b^2}
\|x\|^2_{BMO^c(\Gamma)}.
\end{eqnarray*}

(ii) $T_{t}|P_bx|^2-|T_tP_bx|^2$ exists in $\N$ for any $t,b>0$ and
\begin{eqnarray*}
\|T_{t}|P_bx|^2-|T_tP_bx|^2\|\leq \frac {6t}{b^2}
\|x\|^2_{BMO^c(\Gamma)}.
\end{eqnarray*}

(iii) $P_bx$ weak $^*$ converges to $x$ in $BMO^c(\Gamma)$ as $b\rightarrow0$.

(iv) $x=0$ in $BMO^c(\Gamma)$ iff $P_bx=0$ in $bmo^c(\T)$ for any $b>0$.

The similar properties hold for $y\in BMO^r(\Gamma)$.
\end{lemma}

\begin{proof} Let us fix $b>0$ and a net $x_\lambda \in \N^0$ such that $\Phi(x_\lambda)$ converges with respect to the STOP
topology in $\mathcal{L}$. By (i), we mean that $P_{2b}\Gamma(P_bx_\lambda
,P_bx_\lambda )$ weak$^*$ converges in $\N$ and the limit is with a
norm smaller than $\frac{b^2}6\|x\|_{BMO^c(\Gamma)}^2$. We first
deduce from $\Gamma^2\gl 0$ and Proposition \ref{monot} that
 \begin{align}\label{Gb}
 \frac{b^2}{2} P_{2b}\Gamma(P_bx_\lambda ,P_bx_\lambda )&=   \int_0^b
P_{2b}\Gamma(P_bx_\lambda ,P_bx_\lambda )
  sds\nonumber
   \lel \int_0^b P_{2b}\Gamma(P_{b-s}P_sx_\lambda ,P_{b-s}P_sx_\lambda ) sds\nonumber\\
  &\le \int_0^b P_{3b-s}\Gamma(P_sx_\lambda ,P_sx_\lambda ) sds\nonumber\\
  &\kl 3 \int_0^b P_{b+s}\Gamma(P_sx_\lambda ,P_sx_\lambda ) sds  .
  \end{align}
By Proposition \ref{converge}, $\Phi(x_\lambda)$ converges in the
STOP topology implies that the last term in inequality (\ref{Gb})
weak $^*$ converges in $\N$. Thus
$\Gamma(P_bx_\lambda,P_bx_\lambda)$ weak $^*$ converges in $\N$ and
the limit exists in $\N$ with a norm bounded by $\frac
{54}{b^2}\|x\|^2_{BMO^c(\Gamma)}$. For (ii), we apply lemma
\ref{lemma} (i) and $\Gamma^2\geq0$ and get
\begin{align*}
& T_{t}|P_bx_\lambda|^2-|T_tP_bx_\lambda|^2\lel \int_0^t
T_{t-s}\Gamma(T_sP_bx_\lambda,T_sP_bx_\lambda)ds\\
&\leq  \int_0^t T_{t}P_{\frac {2b}3}\Gamma(P_{\frac
 b3}x_\lambda,P_{\frac
 b3}x_\lambda)ds \lel
 tT_tP_{\frac {2b}3}\Gamma(P_{\frac b3}x_\lambda,P_{\frac b3}x_\lambda).\end{align*}
Applying (\ref{Gb}), we have
\begin{eqnarray*}
T_{t}|P_bx_\lambda|^2-|T_tP_bx_\lambda|^2&\leq&\frac
{54t}{b^2}T_t\int_0^bP_{b+s}\Gamma(P_sx_\lambda ,P_sx_\lambda ) sds
.
\end{eqnarray*}
Thus $T_t|P_bx_\lambda|^2-|T_tP_bx_\lambda|^2$ weak $^*$ converges
in $\N$ and the limit exists in $\N$ with a norm bounded by $\frac
{54t}{b^2}\|x\|^2_{BMO^c(\Gamma)}$ for any $t>0$.

To prove (iii), we use the same idea in the proof of Lemma
\ref{mulgamma}. For any $t>0, 0<b<\min\{t^2,1\}$. Let
$Q_bx=(I-P_b)x=\int_0^b\frac{\partial P_sx}{\partial s}ds$. Then
\begin{eqnarray*}
&&\int_0^\infty P_{t+s} {\Gamma }[P_sQ_b(x)]\min\{t,s\}ds\\
&=&\int_0^\infty P_{t+s}\min\{t,s\}\Gamma [\int_s^{b+s} \frac{
 \partial P_v}{\partial v}xdv]ds\\
\text{(first ineq. of Lemma  \ref{convex gamma})} &\leq& \int_0^\infty \min\{t,s\}\frac1sP_{t+s}\bigg(\int_s^{b+s} {\Gamma }[v\frac{\partial P_v}{\partial v}%
x]dv\bigg)ds \\
\text{(Prop. \ref{monot}) } &\leq&2\int_0^\infty P_{t}\bigg(\int_s^{b+s} {\Gamma }[v\frac{\partial P_v}{%
\partial v}x]dv\bigg)ds \\
\text{(change of variables)} &=&8\int_0^\infty P_{t}\bigg(\int_{\frac s2}^{\frac {b+s}2} {\Gamma }[v\frac{\partial P_v}{
\partial v}P_vx]dv\bigg)ds\\
&=&8\int_0^\infty P_{t}\bigg(\int_{\frac s2}^{\frac {b+s}2} P_v{\Gamma }[P_vx]dv\bigg)ds\\
(\text{Integrate on $ds$ first})&=&8\int_0^\infty P_{t}P_v{\Gamma }[P_vx]\min\{2v,b\}dv\\
&\leq&8\int_{\sqrt b}^\infty P_{t}P_v{\Gamma }[P_vx]b dv+8\int_{0}^{\sqrt b} P_{t}P_v{\Gamma }[P_vx]v dv\\
&\leq&8\sqrt b\int_{\sqrt b}^\infty P_{t+v}{\Gamma }[P_vx] \min\{t,v\} dv+8\int_{0}^{\sqrt b} P_{t+v}{\Gamma }[P_vx]\min\{t,v\} dv.
\end{eqnarray*}
Thus, for any $t>0, g\in L^1_+(\N)$,
\begin{eqnarray*}
&&\tau(g\int_0^\infty P_{t+s} {\Gamma }[P_sQ_b(x)]\min\{t,s\}ds)\\
&\leq& 8\sqrt b\tau (g\int_{\sqrt b}^\infty P_{t+v}{\Gamma }[P_vx] \min\{t,v\} dv)+8\tau(g\int_{0}^{\sqrt b} P_{t+v}{\Gamma }[P_vx]\min\{t,v\} dv).
\end{eqnarray*}
This means
\begin{eqnarray*}
&&\lim_\la\tau(g\int_0^\infty P_{t+s} {\Gamma }[P_sQ_b(x_\lambda)]\min\{t,s\}ds)\\
&\leq& 8\sqrt b\lim_\la\tau (g\int_{\sqrt b}^\infty P_{t+v}{\Gamma
}[P_vx_\lambda] \min\{t,v\} dv)+8\lim_\la\tau(g\int_{0}^{\sqrt b}
P_{t+v}{\Gamma }[P_vx_\lambda]\min\{t,v\} dv).
\end{eqnarray*}
The first term in the right hand side converges to $0$ as
$b\rightarrow0$. We claim the second term converges to $0$ too. If
not,  there exists $\epsilon>0$ such that
$\tau(g\lim_\lambda\int_{0}^{\sqrt b} P_{t+v}{\Gamma
}[P_vx_\lambda]\min\{t,v\} dv)>\epsilon$ for all $b$. We reach a
contradiction with the absolute continuity of integrals  by choosing
$x_{\lambda_0}$ such that $\tau(g\lim_\lambda\int_{0}^{1}
P_{t+v}{\Gamma }[P_v(x_\lambda-x_{\la_0})]\min\{t,v\} dv)<\frac
\epsilon2$. The assertion (iii) is proved.

We now prove (iv). Let $x_\lambda\in \N^0$ be a net weak$^*$
converges to $x$ in $BMO^c(\Gamma)$. Suppose $x=0$ in
$BMO^c(\Gamma)$. The proof of (ii) implies that $w(P_bx_\lambda)$
weakly converges to $0$ in $V_\infty^c$. Here $w$ and $V_\infty^c$
are the embedding and Hilbert $L_\infty$-module defined for the
study of $bmo^c(\T)$ in Section 5.2.3. Therefore, $P_bx_\lambda$
weakly converges to $0$ in $bmo^c(\T)$ for every $b>0$. To prove the
reverse, recall that $P_bx$ weakly converges to $0$ in $bmo^c(\T)$
means that $T_{t}|P_b(x_\lambda)|^2-|T_tP_b(x_\lambda)|^2=\int_0^t
T_{t-s}\Gamma(T_sP_bx_\lambda,T_sP_bx_\lambda)ds$ weak $^*$
converges to $0$ in $\N$ for any $b>0$. Use the same idea as the
proof of (ii), we have $tP_{b}T_t(x_\lambda)$ weakly converges to
$0$ in $BMO^c(\Gamma)$ for any $b,t>0$. Then $P_{2b}(x_\lambda)$
weakly converges to $0$ in $BMO^c(\Gamma)$ for any $b>0$ since
$b^2P_b$ is an average of $tT_t$. This means $P_{2b}x=0$ in
$BMO^c(\Gamma)$ for any $b$. By (iii), we conclude that $x=0$  in
$BMO^c(\Gamma)$.

The same argument works for $BMO^r(\Gamma)$.
\end{proof}

For $x\in BMO^c(\Gamma), y\in BMO^r(\Gamma)$, we say $x=y$ if
$P_{b}(x-y)=0$ in $bmo^c(\T)\cap bmo^r(\T)$ for any $b>0$. For $x\in
BMO^c(\Gamma)$ given, such a $y$ is unique in $BMO^r(\Gamma)$ because of Lemma \ref{lastlemma} (iv).

\begin{defi}Let $BMO(\Gamma)$ be the space of all $x\in BMO^c(\Gamma)$ which belongs to $BMO^r(\Gamma)$ too.
 Define
$$\|x\|_{BMO(\Gamma)}=\max\{\|x\|_{BMO^c(\Gamma)}, \|y\|_{BMO^r(\Gamma)}\}.$$
Here $y$ is the unique element in $BMO^r(\Gamma)$ such that
$P_{b}(x-y)=0$ in $bmo^c(\T)\cap bmo^r(\T)$ for all $b>0$.
\end{defi}

\begin{theorem} Let $(T_t)$ be a standard semigroup satisfying
$\Gamma^2\geq0$ and admitting a reversed Markov dilation with \au
continuous path. Then
 \[ [BMO(\Gamma),L_q^0(\N)]_{\frac{q}{p}}\lel L_p^0(\N) \pl. \]
\end{theorem}

\begin{proof} Let $x\in BMO(\Gamma)\cap L_q^0(\N)$.  By
Proposition \ref{a.u.}, we know that $\pi_{\ttt_a}(x)$ has
continuous path with respect to the filtration $\hat\M_{a,t}$. Note
that
 \[ \pi_{\ttt_a}(x) \lel
 P_{\Gamma}(\pi_{\ttt_a}(x))+P_{br}(\pi_{\ttt_a}(x)) \pl .\]
However, $P_{br}(\pi_{\ttt_a}(x))=\int_0^{\ttt_a}\pi_r(\partial
P_{B_r}(x))dr$ is a stochastic integral against the Brownian motion
and hence has continuous path. Taking the difference, we know that
$P_{\Gamma}(\pi_{\ttt_a}(x))$ has a.u. continuous path. We can now
copy the proof for
 $bmo(\T)$. More precisely, let $h_1(\Gamma)$ be an abstract predual of $BMO(\Gamma)$.
 Similar to the proof of Lemma \ref{swindle}  we
 have $\pi_0^*(P_\Gamma H_1^c(\M))\subset h_1^c(\Gamma)$ since $\pi_0^*(h_1^d(\M))=\{0\}$. Then, by the same argument used in the proof of Theorem
 \ref{intt1}, we have
 $$L_p^0(\N)\subset [L_2^0(\N),h_1(\Gamma)]_{\frac{2-p}{p}},$$ for $1<p<2$. By
 duality and Wolff's theorem, we obtain the result.
\qd


{\bf Open problems.} At the end of this article we want to mention
some open problems.

\medskip

(i) {\it $H^1$-BMO duality for semigroup of operators.} Fefferman's
$H^1$-BMO duality theory has been studied in the context of
semigroups by many researchers. In particular, Varopoulos
established an $H^1$-BMO duality theory for a ``good" semigroups by
a probabilistic approach. Duong/Yan studied this topic for operators
with heat kernel bounds (see \cite{DuLi}). In their proofs, the
geometric structure of Euclidean spaces is essential. Mei (see
\cite{Mei}) provides a first approach of this problem in the context
of von Neumann algebras with two additional assumptions on the
semigroups. The authors expect a more general $H^1$-BMO duality in
the context of semigroups.

\medskip
(ii) {\it Comparison of different semigroup BMO-norms}. There are
several natural semigroup BMO norms as introduced in this article. A
complete comparison of them is in order. In particular, it will be
interesting to investigate the conditions on the semigroups so that
we have the estimates,

(a) $\|\cdot\|_{BMO^c(\P)}\simeq \|\cdot\|_{BMO^c(\partial)}\simeq
\|\cdot\|_{BMO^c(\Gamma)}$.

(b) $\|\cdot\|_{bmo^c(\P)}\simeq \|\cdot\|_{BMO^c(\P)}\simeq
\|\cdot\|_{BMO^c(\hat{\Gamma})}$.

(c) $\|\cdot\|_{bmo^c(\T)}\simeq \|\cdot\|_{BMO^c(\T)}.$

(c')  $\sup_t\|T_tx-T_{2t}x\|\leq c \|x\|_{bmo^c(\T)}$.

(d) $\sup_t\|T_t\int_0^t|\frac{\partial T_sx}{\partial
s}|^2sds\|\leq c \|x\|^2$.

\medskip

\medskip
(iii) The classical BMO functions $\varphi$ on $\rz$ is integrable
with respect to $\frac {1}{1+t^2}dt$. {\it What is a noncommutative
analogue of this property ?} A more precise question is, does there
exist a normal faithful state ${\bf \tau}$ on $\N$ such that ${\bf
\tau}|x| \leq c\|x\|_{BMO(\T)}$ for $x\in \N$.
\medskip

\bibliographystyle{alpha}
\bibliography{bibli}
\end{document}